\documentclass[11pt, a4paper,reqno]{amsart}
\allowdisplaybreaks[1]
\usepackage[latin1]{inputenc}
\usepackage{amsfonts,amsmath,amssymb, amscd,  graphicx, latexsym, color, exscale, enumerate, pstricks}
\DeclareGraphicsExtensions{.jpg,.pdf,.mps,.eps,.png}

\usepackage[all]{xy}
\usepackage{psfrag}

\usepackage{epstopdf}

\newtheorem{thm}{Theorem}[section]
\newtheorem*{thm*}{Theorem}
\newtheorem{lem}[thm]{Lemma}

\newtheorem{prop}[thm]{Proposition}
\newtheorem{cor}[thm]{Corollary}

\theoremstyle{definition}
\newtheorem{example}[thm]{Example}
\newtheorem{rem}[thm]{Remark}
\newtheorem{rems}[thm]{Remarks}

\newtheorem{nota}[thm]{Notation}

\newcommand\bs{{\boldsymbol s}}
\newcommand\bu{{\boldsymbol u}}

\newcommand{\Z}{\mathbb{Z}}

\newcommand{\C}{\mathbb{C}}

\newcommand{\ot}{\otimes}

\newcommand{\s}{\sigma}

\numberwithin{equation}{section}

\title[Inductive basis on BMW algebras \& (transverse) Markov traces]{Inductive basis on Birman-Murakami-Wenzl algebras and (transverse) Markov traces}

\author{Lo\"ic Poulain d'Andecy, Anne-Laure Thiel and Emmanuel Wagner}

\date{06/11/2017}

\begin{document}

\begin{abstract}
We construct a new inductive basis of the Birman-Murakami-Wenzl algebra. Using it, we provide a new proof of the existence of the Markov trace on the BMW algebras affording the two-variable Kauffman polynomial. We prove also that all the transverse Markov traces on the BMW algebras are determined  by the self-linking number, the HOMFLY--PT polynomial and the two-variable Kauffman polynomial.
\end{abstract}

\maketitle

\tableofcontents


\section*{Introduction}

There are two fundamental theorems relating links and braids. The first by Alexander asserts that any link can be presented as the closure of a braid; the second by Markov specifies  how to pass from one braid to another if they have isotopic closures. Since the seminal work of Jones \cite{Jo2}, it is understood that one can use effectively these two main theorems relating links and braids to produce link invariants.\\

The Jones construction uses as a key ingredient a finite dimensional quotient of the group algebra of the braid group, namely the Hecke algebra. He constructs explicitly a Markov trace using a very specific basis. This Markov trace affords the HOMFLY--PT link polynomial \cite{HOMFLY}, \cite{PT}. There are other proofs of the existence of this invariant; for instance, diagrammatic ones using the skein relations or representation theoretical ones using $R$-matrices and the natural representation of the quantum group of type $A$ \cite{Tu}.\\
 
A remarkable feature is that the Hecke algebra or Iwahori-Hecke algebra appeared in  other contexts before and is in particular a flat deformation of the group algebra of the symmetric group, see \cite{KaTu} or \cite{GePf} for an overview. In fact the basis used to construct the Markov trace is just the positive lift of the natural basis of the group algebra of the symmetric group (it amounts to see a reduced word in the symmetric group as a word in the braid group).\\

The case of the Birman-Murakami-Wenzl \cite{BW,Mu} algebra is quite different. It was introduced after the development of quantum link invariants and was precisely constructed to have a Markov trace affording the two-variable Kauffman polynomial \cite{Kau2}. The existence of this invariant was proved using skein theory by Kauffman and was from the very beginning seen as a two variable version of the infinite families of the quantum invariants corresponding to the natural representation of quantum groups of type $B/C/D$ (which also provides a proof of its existence, see \cite{Tu}, \cite{MorBMW}). In addition the BMW algebra is a flat deformation of the Brauer algebra.\\ 

Our first result is to provide a proof of the existence of the two-variable Kauffman polynomial similar to the one given by Jones for the HOMFLY--PT polynomial. To achieve this, we first construct a nice inductive basis of the BMW algebra.\\

Let $B_n$ be the braid group on $n$ strands. We denote by $BMW_n$ the BMW algebra seen as a quotient of the group algebra of $B_n$ and by $s_i$, $i=1,\cdots,n-1,$ its generators, images of the standard ones of $B_n$. The basis $\mathbf{b}^{BMW}_{n}$ is constructed inductively starting from the basis $\mathbf{b}^{BMW}_{1}= \{ 1 \}$ of $BMW_{1}$. Consider the following elements in $BMW_{n+1}$:
$$ x_{n-1,i} = s_{n-1}^{-1} \dots s_i^{-1} \quad \mbox{and} \quad y_{i,n-1} = s_i \dots s_{n-1}.$$
We denote by $\iota_n$ the natural embedding of $B_n$ in $B_{n+1}$ given by adding a vertical strand on the right.

\begin{thm*} \label{basisBMW}
The family 
$$\mathbf{b}^{BMW}_{n+1} = \{ \iota_n(b), \iota_n(b)x_{n,i}, y_{i,n}\iota_n(b) | b \in \mathbf{b}^{BMW}_{n}, i = 1, \dots, n \}$$
 forms a basis of $BMW_{n+1}$.
\end{thm*}

As said before we obtain as a consequence a new proof of the existence of the Kauffman polynomial, see Theorem \ref{reconstr Kauffman}. We expect this proof to be well suited to be adapted to study similar questions about traces on other BMW algebras \cite{CoWa} than the classical one studied in this paper.

We finish this discussion on classical Markov traces on BMW algebras by mentioning that, in the literature  regarding the question of classification of Markov traces on the BMW algebras, the trace is always supposed to satisfy an additional property that we call the multiplicative property, see for instance \cite{JoSubKnot}. This property is very natural from the tangle description of the BMW algebras or from the quantum groups perspective where it just amounts to saying that the invariant evaluated on a disjoint union of links is just the product of the invariants of each part. It was proved in \cite{MaWaBMW} that, even without supposing this additional property, the space of Markov traces on the BMW algebras is two dimensional and spanned by the traces affording the HOMFLY--PT polynomial and the two variable Kauffman polynomial.\\

The fact that one could a priori remove the multiplicativity condition on the Markov trace comes from the fact that the BMW algebra is a cubical quotient of the group algebra of the braid group. For the Hecke algebra this property follows directly from the stabilization and the quadratic relation defining the Hecke algebra.
This suggested that on the BMW algebra there could be enough room to construct other types of Markov traces, namely transverse Markov traces. Studying cubical quotients of the group algebra of the braid group to produce effective transverse link invariants  was also the motivation of Bellingeri and Funar \cite{BeFu} and pursued by Orevkov \cite{Orev}.\\

For a detailed introduction on transverse braids and transverse links, one can consult the nice survey by Etnyre \cite{EtnySur}. For us the key result is the Markov type theorem \cite{OrSh,Wrin}. In a nutshell, given the standard contact structure on the three-sphere one can consider links in the three sphere which are transverse to the plane distribution up to isotopies which preserve this property. The main invariant is the self-linking number which for a link given as a braid closure is just the difference between the braid index and the algebraic length of the braid. One important question was to see if the topological type of a transverse knot together with its self-linking number is a complete transverse invariant. This was proved to be true for certain classes of knots \cite{BiWrin}, \cite{Eliash,Etny1} but also to be false in general \cite{BiMen,Etny2,Etny3}.\\
 
Now there are other ways to prove negative results by using effective transverse invariants which can distinguish transverse links with the same self-linking number and the same topological type \cite{Ng,NgOzThu}. Unfortunately transverse Markov traces on the BMW algebras do not produce such effective invariants as the next theorem asserts.\\

\begin{thm*}
Transverse Markov traces on the BMW algebras are determined by the HOMFLY--PT polynomial, the two-variable Kauffman polynomial and the self-linking number.
\end{thm*}

The classification of the transverse Markov traces on the Hecke algebras was surely known by experts, but we could not find a written reference. We obtain in fact the classification of the transverse Markov traces on the BMW algebras: the only transverse Markov traces on the BMW algebras are the transverse ones coming from the Hecke algebras and the topological one giving the two-variable Kauffman polynomial (a classical invariant is always a transverse invariant). A nice interpretation of the former result passes through the famous Morton-Franks-Williams inequality \cite{FrWi} \cite{MorIne}. This inequality says that the minimal degree in the framing variable $a$ of the HOMFLY--PT polynomial is greater or equal to the self-linking number. The HOMFLY--PT invariant is a Laurent polynomial in $a$ and the MFW inequality allows to say that the HOMFLY--PT polynomial multiplied by $a$ to the power the opposite of the self-linking number is in fact a polynomial in $a$, one can then divide the result following the increasing powers of the value of the HOMFLY--PT polynomial of the unknot. The polynomials in front of the powers of the value of the unknot are immediately transverse invariants and in fact provide the basic transverse Markov traces on the Hecke algebras.\\

The previous interpretation of the theorem indicates that one can deduce from Morton-Franks-Williams type inequalities for topological link invariants existence of non-effective transverse Markov traces. It applies in particular to colored HOMFLY--PT polynomials \cite{WuIne}. In another direction, there are some other finite dimensionnal quotients of the group algebra of the braid group, on which one can try to classify transverse Markov traces. It could be achieved on the extension of the BMW algebras defined by Marin and Wagner \cite{MaWaBMW} or on the Yokonuma-Hecke algebra \cite{Yoko} \cite{PdAWyoko} but would certainly produce non-effective transverse invariants. Another candidate is the defining algebra of the Links-Gould invariant studied in \cite{MaWa}, but the understanding of this algebra, and in particular of inductive basis would require more work to attack the classification of transverse Markov traces on it.\\

In the first section, we introduce the Hecke and BMW algebras and their properties. We describe and prove the existence of the new inductive basis on the BMW algebras. In the second section, after introducing classical Markov traces, we provide a new proof of the existence of the two-variable Kauffman polynomial link invariant. In the third and last section we investigate the question of transverse Markov traces on the Hecke and BMW algebras. This section does also contain a discussion and a proof of the Morton-Franks-Williams inequality.\\

\section{The Hecke and Birman-Murakami-Wenzl algebras}

Fix $n\geq 2$ and the ring $R = \C \left[ a^{\pm 1}, q^{\pm 1}, (q-q^{-1})^{\pm 1}\right] $. 

We start this Section by recalling some definitions and well-known facts about these two algebras $H_n$ and $BMW_n$, which are finite-dimensional quotients of the group algebra of the braid group $B_n$ on $n$ strands. In the end, we will construct a new basis of the latest. We refer to \cite{KaTu} and \cite{JoSubKnot} for a more detailed treatment of these algebras.

Set $H_1 \cong R \cong BMW_1$.

\subsection{The Hecke algebras}

The Hecke algebra $H_n$ is the algebra over $R$ generated by the elements $\sigma_i$ for $i= 1,\dots, n-1$ satisfying the braid relations:
\begin{align}
\s_i\s_j & =  \s_j\s_i \qquad & &\text{for distant $i,j=1,\dots,n-1$,} \label{H1}\\ 
\s_i\s_{i+1}\s_i & =  \s_{i+1}\s_i\s_{i+1} & & \text{for $i=1,\dots,n-2$,} \label{H2}
\end{align}
and
\begin{equation}
\s_i^2=1+ (q-q^{-1})\s_i,\ \ \ \ \quad\text{for $i=1,\dots,n-1$.} \label{H3}
\end{equation}
It follows that:
\begin{equation}
\s_i^{-1}=\s_i-(q-q^{-1}),\ \ \ \ \quad\text{for $i=1,\dots,n-1$,} \label{H4}
\end{equation}
or equivalently
\begin{equation}\label{H5}
0= \frac{\s_i^{-1} -\s_i}{q - q^{-1}} + 1,\ \ \ \ \quad\text{for $i=1,\dots,n-1$.} 
\end{equation}

The Hecke algebra $H_n$ can be seen as a flat deformation of the group algebra of the symmetric group $S_n$ and, as such, its dimension is equal to the order $ n!$ of $S_n$.

\subsection{The Birman-Murakami-Wenzl algebras }

The Birman-Murakami-Wenzl algebra $BMW_{n}$ (or BMW algebra for short) is the algebra over $R $ generated by the elements $s_i$ for $i= 1, \dots , n-1$ satisfying the braid relations:
\begin{align}
s_i s_j & =  s_j s_i & & \mbox{for distant} \ i,j=1,\dots, n-1 \label{BMW1}\\
s_i s_{i+1} s_i& =  s_ {i+1} s_i s_{i+1}& & \mbox{for} \ i=1,\dots, n-2   \label{BMW2}
\end{align}
and, together with the elements 
\begin{equation}
\label{defei}
e_i = \frac{s_i^{-1} -s_i}{q - q^{-1}} + 1,
\end{equation}
the following additional relations:
\begin{align}
e_i s_i & =  a^{-1} e_i & & \mbox{for} \ i=1,\dots, n-1 \label{BMW3}\\
e_i s_{i+1} e_i & = a e_{i} & & \mbox{for} \ i=1,\dots, n-2   \label{BMW4}\\
e_i s_{i+1}^{-1} e_i & = a^{-1} e_{i} & & \mbox{for} \ i=1,\dots, n-2 .  \label{BMW4B}
\end{align}
Denote by $\overline{\cdot}$ the ring automorphism of $R$ defined by $\overline{a}=a^{-1}$ and $\overline{q}=q^{-1}$. The following map defines a ring anti-automorphism $\eta_n :BMW_n\to BMW_n$:
\begin{equation}\label{eta}
\eta_n : s_i\mapsto s_i^{-1}\ \ \ \ \ \text{and}\ \ \ \ \ \ \alpha\mapsto\overline{\alpha},\ \forall \alpha\in R .
\end{equation}
Using the conjugation by the half-twist $ s_{i+1} s_i s_{i+1}$, one also has:
\begin{align}
e_{i+1} s_{i} e_{i+1} & = a e_{i+1} & & \mbox{for} \ i=1,\dots, n-2   \label{BMW4'}\\
e_{i+1} s_{i}^{-1} e_{i+1} & = a^{-1} e_{i+1} & & \mbox{for} \ i=1,\dots, n-2 .  \label{BMW4B'}
\end{align}
From these relations follow some other ones that we will use repeatedly in the sequel:
\begin{align}
e_i e_{i+1} e_i & = e_{i} & &    \label{BMW5}\\
e_{i+1} e_{i} e_{i+1} & = e_{i+1} & &   \label{BMW6} 
\end{align}
for $i=1,\dots, n-2$, and
\begin{align}
s_i - s_i^{-1} & = \left( q - q^{-1}\right) \left( 1 - e_i \right) & &  \label{BMW7}\\
s_i^{2} & = 1 + \left( q - q^{-1}\right)s_i - \left( q - q^{-1}\right)a^{-1} e_i & &  \label{BMW8}\\
s_i^{2} & = \left( 1 - \left( q - q^{-1}\right)a^{-1} \right)1 +\left( a^{-1} + \left( q - q^{-1}\right) \right)s_i - a^{-1} s_i^{-1} & &  \label{BMW9}\\
s_i^{-2} & = 1 - \left( q - q^{-1}\right)s_i^{-1} + \left( q - q^{-1}\right)a e_i & &  \label{BMW10}\\
s_i^{-2} & = \left( 1 + \left( q - q^{-1}\right)a \right)1 + \left( a - \left( q - q^{-1}\right) \right) s_i^{-1} -a s_i & & \label{BMW11}\\
e_i^2 & = \left( \frac{a - a^{-1}}{q - q^{-1}} +1 \right) e_{i} & &    \label{BMW12}
\end{align}
for $i=1,\dots, n-1$.

The BMW algebra $BMW_n$ can be seen as a flat deformation of the Brauer algebra $Br_n$ with parameter $\delta$ and, as such, its dimension is equal to $ \frac{(2n)!}{2^n n!} = (2n-1) . (2n-3) \dots 5 . 3 . 1$. Note that a $\C \left[ \delta \right]-$basis $\mathbf{b}^{Br}_n$ of the Brauer algebra is given by all the pairings of a set of $2n$ points. 

As an algebra over $\C \left[ \delta \right]$, the Brauer algebra $Br_n$ is simply isomorphic to the specialization at $a = q= 1$ of 
$$\C \left[\delta, a^{\pm 1}, q^{\pm 1}\right] / \left( (\delta -1 ) (q - q^{-1}) - (a - a^{-1}) \right) \otimes_{\C \left[ a^{\pm 1}, q^{\pm 1}\right]} KT_n,$$
where $KT_n$ is the Kauffman tangle algebra. The algebra $KT_n$ is the free algebra over $\C \left[ a^{\pm 1}, q^{\pm 1}\right]$ generated by $(n,n)$--framed unoriented tangles up to regular isotopy and the following local skein relations:
\begin{equation*}
 \begin{array}{c} \vspace{-0.2cm} \includegraphics[height=0.65cm]{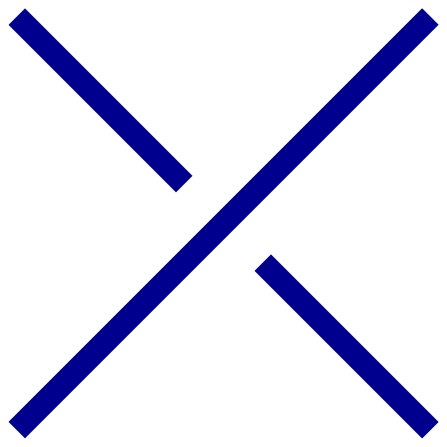} \end{array} - \begin{array}{c} \vspace{-0.2cm}  \includegraphics[height=0.65cm]{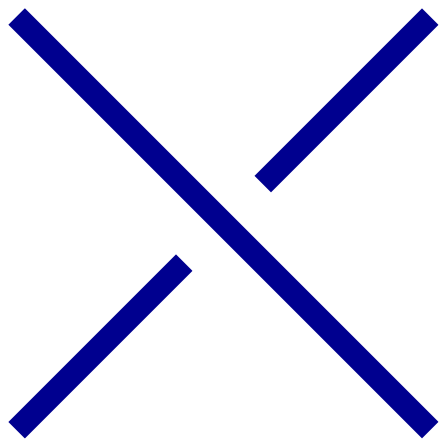} \end{array} = \left(q - q^{-1}\right) \left( \begin{array}{c} \vspace{-0.2cm} \includegraphics[height=0.65cm]{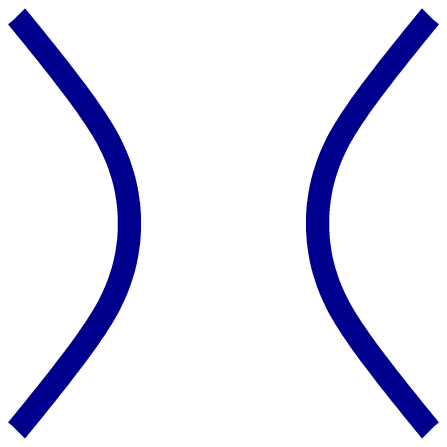} \end{array} - \begin{array}{c} \vspace{-0.2cm} \includegraphics[height=0.65cm]{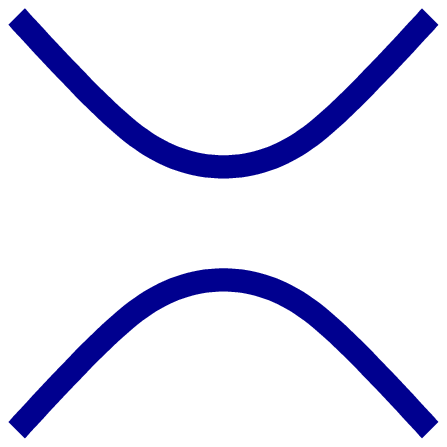} \end{array} \right),
\end{equation*}
\begin{equation}
\label{skeinrel}
 \begin{array}{c} \vspace{-0.2cm} \includegraphics[height=0.65cm]{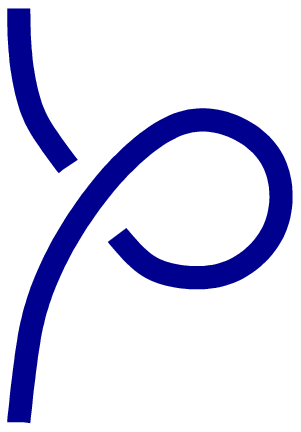} \end{array} = a \begin{array}{c}  \vspace{-0.2cm} \includegraphics[height=0.65cm]{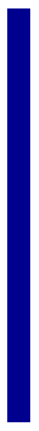} \end{array} \quad \mbox{and} \quad \begin{array}{c} \vspace{-0.2cm} \includegraphics[height=0.65cm]{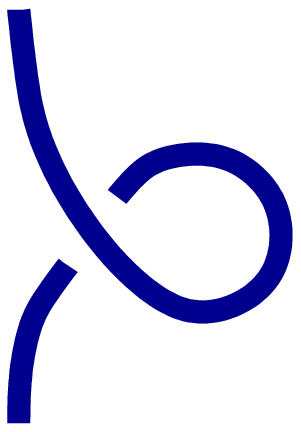} \end{array} = a^{-1} \begin{array}{c} \vspace{-0.2cm} \includegraphics[height=0.65cm]{one.eps} \end{array}.
\end{equation}
From those we can deduce
\begin{eqnarray*}
 \begin{array}{c} \vspace{-0.2cm} \includegraphics[height=0.65cm]{ucurl-l.eps} \end{array} - \begin{array}{c} \vspace{-0.2cm} \includegraphics[height=0.65cm]{ucurl-r.eps} \end{array} & = & \left(q - q^{-1}\right) \left( \begin{array}{c} \vspace{-0.2cm} \includegraphics[height=0.65cm]{one.eps} \end{array} \begin{array}{c} \vspace{-0.2cm} \includegraphics[height=0.65cm]{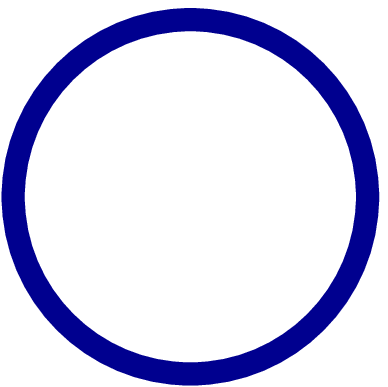} \end{array} - \begin{array}{c} \vspace{-0.2cm} \includegraphics[height=0.65cm]{one.eps} \end{array} \right) \\
\left(a - a^{-1}\right) \begin{array}{c} \vspace{-0.2cm} \includegraphics[height=0.65cm]{one.eps} \end{array}  & = & \left(q - q^{-1}\right) \left( \begin{array}{c} \vspace{-0.2cm} \includegraphics[height=0.65cm]{one.eps} \end{array} \begin{array}{c} \vspace{-0.2cm} \includegraphics[height=0.65cm]{loop.eps} \end{array} - \begin{array}{c} \vspace{-0.2cm} \includegraphics[height=0.65cm]{one.eps} \end{array} \right)
\end{eqnarray*}
and hence 
\begin{equation}
 \begin{array}{c} \vspace{-0.2cm} \includegraphics[height=0.65cm]{one.eps} \end{array} \begin{array}{c} \vspace{-0.2cm} \includegraphics[height=0.65cm]{loop.eps} \end{array} =\left( \frac{a - a^{-1}}{q - q^{-1}} +1 \right) \begin{array}{c} \vspace{-0.2cm} \includegraphics[height=0.65cm]{one.eps} \end{array} 
\label{valeurtrivial}
\end{equation}
if we extend $\C \left[ a^{\pm 1}, q^{\pm 1}\right]$ to $R$.

It is a well-known fact, see \cite{MorBMW}, that the BMW algebra $BMW_n$ is isomorphic to the Kauffman tangle algebra $KT_n$ (up to the extension of scalars $R \otimes_{\C \left[ a^{\pm 1}, q^{\pm 1}\right]}- $) via:
\begin{eqnarray*}
BMW_n & \rightarrow & KT_n \\
s_i & \mapsto & \begin{array}{c} \vspace{-0.1cm} \includegraphics[height=0.65cm]{Xing-p.eps} \end{array} \\
s_i^{-1} & \mapsto & \begin{array}{c} \vspace{-0.1cm} \includegraphics[height=0.65cm]{Xing-n.eps} \end{array} \\
e_i & \mapsto & \begin{array}{c} \vspace{-0.1cm} \includegraphics[height=0.65cm]{cupcap.eps} \end{array}
\end{eqnarray*}

\begin{nota}
From now on, let us set 
$$ z = q - q^{-1}, \quad \delta = \dfrac{a - a^{-1}}{q - q^{-1}} +1 \quad \text{and} \quad \delta^H = \dfrac{a - a^{-1}}{q - q^{-1}}.$$
\end{nota}

Topologically, it is obvious that, for $i= 1, \dots, n-2$,
\begin{equation}
\label{BMW13}
s_{i+1} e_i s_{i+1} = s_i^{-1} e_{i+1} s_i^{-1}
\end{equation}
and
\begin{equation}
\label{BMW14}
s_{i+1}^{-1} e_i s_{i+1}^{-1} = s_i e_{i+1} s_i
\end{equation}
from which we can deduce, using the definition \eqref{defei} of the $e_i$'s, that
\begin{equation}
\label{BMW15}
s_{i+1} s_i^{-1} s_{i+1} = -z s_{i+1}^2 + s_{i+1} s_i s_{i+1} - s_i^{-1} s_{i+1} s_i^{-1} + s_i^{-1} s_{i+1}^{-1} s_i^{-1} + z s_{i}^{-2}
\end{equation}
and, applying the anti-automorphism $\eta_n$ from (\ref{eta}),
\begin{equation}
\label{BMW16}
s_{i+1}^{-1} s_i s_{i+1}^{-1} = z s_{i+1}^{-2} + s_{i+1}^{-1} s_i^{-1} s_{i+1}^{-1} - s_i s_{i+1}^{-1} s_i + s_i s_{i+1} s_i - z s_{i}^{2}.
\end{equation}

\subsection{Common properties and interplay}\label{interplay}

We will now list a few properties that the Hecke and BMW algebras are sharing. So here we will make some statements for a generic algebra $A_n$ with generators $g_i$'s, which are true when $A_n$ is equal to $H_n$ (and hence $g_i = \s_i$) or to $BMW_n$ (and hence $g_i = s_i$).

The family $\left( A_n \right)_{n \geq 1}$ naturally forms a tower of algebras. Let us denote by $\iota_n : A_n \rightarrow A_{n+1}$ the algebra embedding that sends $g_i$ to $g_i$ for $i = 1, \dots, n-1$.

Let us consider, for any $n \geq 2$, the morphism $sh: \iota_n \left( A_n \right) \rightarrow A_{n+1}$ defined by $sh(g_{i_1}^{\epsilon_1} \cdots g_{i_k}^{\epsilon_k} )= g_{i_1 +1}^{\epsilon_1} \cdots g_{i_k+1}^{\epsilon_k}$. Note that for any $u \in \iota_n \left( A_n \right)$ the elements $u$ and $sh(u)$ are conjugate in $A_{n+1}$. More generally, for any $m=1, \dots, n-2$, let $sh^m: \iota_{n-1} \dots \iota_{n-m} \left( A_{n-m} \right) \rightarrow A_{n}$ be the morphism defined by $sh^m(g_{i_1}^{\epsilon_1} \cdots g_{i_k}^{\epsilon_k} )= g_{i_1 +m}^{\epsilon_1} \cdots g_{i_k+m}^{\epsilon_k}$.

Let $n_1,n_2\in\Z_{>0}$. The algebra $A_{n_1}\otimes A_{n_2}$ is canonically a subalgebra of $A_{n}$ with $n=n_1+n_2$: $A_{n_1}$ is identified with the subalgebra $ \iota_{n-1} \dots \iota_{n_1} \left( A_{n_1} \right) $ of $A_n$ generated by $g_1,\dots,g_{n_1-1}$, while $A_{n_2}$ is identified with the subalgebra $ sh^{n_1}\iota_{n-1} \dots \iota_{n_2} \left( A_{n_2} \right) $ of $A_n$ generated by $g_{n_1+1},\dots,g_{n_1+n_2-1}$. These two subalgebras commute. Note finally that two elements $v_1 \ot v_2 \in A_{n_1}\otimes A_{n_2}$ and $v_2 \ot v_1 \in A_{n_2}\otimes A_{n_1}$ are conjuguate in $A_n$. Indeed $v_1 \ot v_2 = u(v_2 \ot v_1)u^{-1}$ with $u = (g_{n_1} \dots g_{n-1})(g_{n_1-1} \dots g_{n-2}) \dots (g_{1} \dots g_{n_2})$ as a direct consequence of the braid relations.

The Hecke algebra $H_{n+1}$ can be decomposed as $R$-module as follows:
\begin{equation}
H_n \oplus (H_n \ot_{H_{n-1}} H_n) \cong H_{n+1}.
\label{decHecke}
\end{equation}
The isomorphism is given by $(u, v\ot w) \mapsto \iota_n(u) + \iota_n(v) \s_n \iota_n(w)$. 

Note that here the behaviors of our two algebras start to deviate as such a decomposition does not exist for $BMW_n$. Its existence is obstructed by relations of the type $e_1s_2e_1 = a e_1$. This is a reason why the study of Markov traces on the BMW algebras is more complicated than the one on the Hecke algebras, and explains the importance of the bases of the BMW algebras constructed in the following subsection.

Finally let us recall how these two algebras are related to one another. The Hecke algebra $H_n$ is the quotient of the BMW algebra $BMW_n$ by the ideal $E$ generated by $\{ e_i =0, i=1, \dots, n-1\}$. The surjection is given by:
\begin{eqnarray}
q_E \, : \,BMW_n & \rightarrow & H_n \\ \label{surjBMWH}
s_i & \mapsto & \s_i. \nonumber
\end{eqnarray}
Actually it is enough to quotient by the single relation $e_1=0$, indeed it implies the others as $e_{i+1} = s_i s_{i+1} e_i s_{i+1}^{-1} s_i^{-1}$.

\subsection{Inductive bases of the BMW algebras}\label{subsec-bases}
Let us recall the following standard construction of an inductive basis $\mathbf{b}^{H}_{n+1}$ for the Hecke algebra $H_{n+1}$, starting from the basis $\mathbf{b}^{H}_{1}= \{ 1 \}$ of $H_{1}$:
\[ \{ \iota_n(b), \iota_n(b) \sigma_n \dots \sigma_i| b \in \mathbf{b}^{H}_{n}, i = 1, \dots, n \}.\]

Similarly, we will now construct inductively a basis $\mathbf{b}^{BMW}_{n+1}$ of $BMW_{n+1}$ starting from the basis $\mathbf{b}^{BMW}_{1}= \{ 1 \}$ of $BMW_{1}$. For $i = 1, \dots, n$, consider the following elements in $BMW_{n+1}$:
$$ x_{n,i} = s_n^{-1} \dots s_i^{-1} \quad \mbox{and} \quad y_{i,n} = s_i \dots s_n.$$

\begin{prop} \label{basisBMW}
The family 
$$\mathbf{b}^{BMW}_{n+1} = \{ \iota_n(b), \iota_n(b)x_{n,i}, y_{i,n}\iota_n(b) | b \in \mathbf{b}^{BMW}_{n}, i = 1, \dots, n \}$$
 forms a basis of $BMW_{n+1}$.
\end{prop}

\begin{rem}
In the sequel we will most of the time simply identify an element $b$ with its image $ \iota_n(b)$ and omit to write $\iota_n$. 
\end{rem}

\begin{proof}
Since this family $\mathbf{b}^{BMW}_{n+1}$ has the good cardinality, we only need to prove that it generates $BMW_{n+1}$. We will prove it recursively. It is true for $\mathbf{b}^{BMW}_{1} = \{ 1 \}$.

In \cite{MorBMW}, it is shown that one can obtain a basis of $BMW_{n+1}$ from the basis $\mathbf{b}^{Br}_{n+1}$ of the Brauer algebra $Br_{n+1}$. Diagrammatically, one just lifts a pairing in $\mathbf{b}^{Br}_{n+1}$ to a tangle by replacing its crossings by indifferently positive or negative crossings. Morton proved that any family of tangles obtained in this way forms a basis of $BMW_{n+1}$.

Hence we just need to show that for any pairing in $\mathbf{b}^{Br}_{n+1}$, there exists a tangle which lifts this pairing and which can be written as a sum of elements of $\mathbf{b}^{BMW}_{n+1}$. Since we will do some diagrammatical reasoning, let us precise that one will read diagrams from bottom to top and that the corresponding element of the BMW algebra is read from left to right. Take a pairing $p$ in $\mathbf{b}^{Br}_{n+1}$, we will distinguish three cases:
\begin{itemize}
\item If the $n+1$st bottom point of $p$ is paired with the $i$th top point for some $i \in \{ 1, \dots , n+1\}$, then we choose a tangle lift of $p$ so that this strand is above the rest of the diagram. This means that this lift can be written $b$ if $i = n+1$ or $b x_{n,i}$ otherwise, with $b$ in $BMW_n$ so in particular, by induction, $b$ can be written as a sum of elements of $\mathbf{b}^{BMW}_{n}$. Thus we have constructed a lift of $p$ that is a sum of elements of $\mathbf{b}^{BMW}_{n+1}$.
\item If the $n+1$st top point of $p$ is paired with the $i$th bottom point for some $i \in \{ 1, \dots , n\}$, we do the exact same reasoning as in the previous case. Pick a tangle lift of $p$ for which this strand is above the rest of the diagram, i.e. it can be written $y_{i,n} b$, with $b \in BMW_n$ being by induction a sum of elements of $\mathbf{b}^{BMW}_{n}$ so that this lift of $p$ is expressed as a sum of elements of $\mathbf{b}^{BMW}_{n+1}$.
\item Otherwise, this means that the $n+1$st bottom point of $p$ is paired with the $i$th bottom point and the $n+1$st top point of $p$ is paired with the $j$th top point for some $i,j \in \{ 1, \dots , n\}$. Let us then lift $p$ to a tangle for which these two strands are above the rest of the diagram, meaning that this lift $t$ is of the form $s_i \dots s_{n-1} b e_n s_{n-1}^{-1} \dots s_{j}^{-1}$ with $b$ in $BMW_{n-1}$. By definition \eqref{defei} of $e_n$, we then have:
\begin{eqnarray*}
t & = & s_i \dots s_{n-1} b s_{n-1}^{-1} \dots s_{j}^{-1} + z^{-1} s_i \dots s_{n-1} b s_n^{-1} s_{n-1}^{-1} \dots s_{j}^{-1} \\
& & - z^{-1} s_i \dots s_{n-1} b s_n s_{n-1}^{-1} \dots s_{j}^{-1} \\
& = & s_i \dots s_{n-1} b s_{n-1}^{-1} \dots s_{j}^{-1} + z^{-1} y_{i,n-1} b x_{n,j} - z^{-1} y_{i,n} b x_{n-1,j} .\\
\end{eqnarray*}
This allows us to conclude using the induction hypothesis on $b \in BMW_{n-1}$ (as $b$ being a sum of elements of $\mathbf{b}^{BMW}_{n-1}$ implies that the last two terms are a sum of elements of $\mathbf{b}^{BMW}_{n+1}$) and on the first term which belongs to $BMW_n$.
\end{itemize}
\end{proof}

\begin{rems}
We observe the following:
\begin{itemize}
	\item[$\bullet$] All these basis elements are Mikado braids \cite{DiGo}.
  \item[$\bullet$] This basis does not specialize to a basis of the Brauer algebra $Br_n$ at $a=q=1$. 
\item If we apply to the basis $\mathbf{b}^{H}_{n+1}$ the algebra anti-automorphism $\sigma_i\mapsto\sigma_i$, or the ring anti-automorphism analogous to $\eta_{n+1}$ in (\ref{eta}), or both, we obtain three other inductive basis of $H_{n+1}$ of the following form:
$$ \{ \iota_n(b), \sigma_i \dots\sigma_n\iota_n(b)| b \in \mathbf{b}^{H}_{n}, i = 1, \dots, n\},$$
$$\{ \iota_n(b), \sigma^{-1}_i\dots\sigma^{-1}_n\iota_n(b)| b \in \mathbf{b}^{H}_{n}, i = 1, \dots, n\},$$
$$\{ \iota_n(b), \iota_n(b)\sigma^{-1}_n\dots\sigma^{-1}_i| b \in \mathbf{b}^{H}_{n}, i = 1, \dots, n\}.$$
\item We remark that the basis $\mathbf{b}^{BMW}_{n+1}$ of $BMW_{n+1}$ is invariant by the ring anti-automorphism $\eta_{n+1}$ in (\ref{eta}). If we apply the algebra anti-automorphism $s_i\mapsto s_i$, we obtain another inductive basis, also invariant by $\eta_{n+1}$, of the form
\[ \{ \iota_n(b), s_i^{-1}\dots s_n^{-1}\iota_n(b), \iota_n(b)s_n\dots s_i| b \in \mathbf{b}^{BMW}_{n}, i = 1, \dots, n\} .\]
\end{itemize}
\end{rems}

\begin{nota}
In the sequel, we will denote
$$ x_{j,i} = s_j^{-1} \dots s_i^{-1} \quad \mbox{and} \quad y_{i,j} = s_i \dots s_j$$
for $1 \leq i \leq j \leq n$. 

We also use the convention that, if $ i > j$, than $x_{j,i} = y_{i,j} = 1$.
\end{nota}

\section{Classical Markov traces}

\subsection{Definition}

A Markov trace on a tower of algebras $\left( A_n \right)_{n \geq 1}$ over $R$ is a family $(t_n)_{n\geq 1}$ of $R$--linear maps $t_n : A_n \rightarrow R$ satisfying the following conditions:
\begin{itemize}
\item trace property: $t_n(uv) = t_n(vu)$,
\item stabilization property: $t_{n+1} \left( \iota_n(u) g_n^{\pm 1} \iota_n(v) \right) = a^{\pm 1} t_n(uv)$.
\end{itemize}
for all $u,v \in A_n$ and $n\geq 1$. If furthermore, the family $(t_n)_{n\geq 1}$ satisfies the
\begin{itemize}
\item multiplicativity property: $t_{n_1+n_2} \left( u_1 \ot u_2 \right) = t_{n_1} \left( u_1 \right)t_{n_2} \left(u_2 \right)$, for all $u_1 \in A_{n_1}, \ u_2 \in A_{n_2}$,
\end{itemize}
the Markov trace is said multiplicative.

We can also consider the
\begin{itemize}
\item inclusion property: $t_{n+1} \left( \iota_n(u) \right) = d t_n(u)$, for some $d \in R$,
\end{itemize}
which is weaker as any multiplicative trace satisfies the inclusion property with $d=t_1(1)$. But for $d=t_1(1)$ and $A_n = H_n$ or $BMW_n$, they are equivalent. To prove the converse in this case, we consider two basis elements $u_1 \in A_{n_1}$ and $u_2 \in A_{n_2}$ with $n = n_1 + n_2$ (and it is enough to do so since a trace is linear). There exist $\epsilon \in \{-1,0, 1\}$ and $v_2, w_2 \in A_{n_2-1}$ such that $u_2 = v_2g_{n-1}^{\epsilon}w_2$. This is a direct consequence of \eqref{decHecke} for the Hecke algebra and of Proposition \ref{basisBMW} for the BMW algebra. Then we proceed by induction. The two properties are clearly equivalent when $n_1 + n_2 =2$, and for $n>2$, it follows immediately from the induction hypothesis and the stabilization property (if $\epsilon = \pm 1$) or inclusion property (if $\epsilon = 0$) that
\begin{eqnarray*}
t_n(u_1 \ot u_2) & =& d^{\overline{1-\epsilon}}a^{\epsilon}t_{n-1}(u_1 \ot v_2w_2)\\
& = & d^{\overline{1-\epsilon}}a^{\epsilon} t_{n_1} \left( u_1 \right)t_{n_2-1} \left(v_2w_2 \right)\\
& = & t_{n_1} \left( u_1 \right)t_{n_2} \left(u_2 \right)
\end{eqnarray*}
where $\overline{1-\epsilon}$ is the remainder of $1 - \epsilon$ modulo $2$.

Observe that while the set of traces on $\left( A_n \right)_{n \geq 1}$ is an $R$-module, the subset of multiplicative traces is not a submodule of it.

\subsection{Markov traces and HOMFLY--PT polynomial}\label{subsec-HOMFLY}

It is a well-known fact that the space of Markov traces on the tower of algebras $\left( H_n \right)_{n \geq 1}$ is of dimension one, see \cite{Jo2} or \cite[Chap 4.5]{GePf}.

Indeed, one observes, from the decomposition \eqref{decHecke}, that the trace of any element of $H_{n+1}$ can be computed by descending induction. For any $u,v,w \in H_n$, one has $t_{n+1}(v\sigma_n w ) = a t_n (vw)$ and, by \eqref{H5}, 
$$t_{n+1}(u) = t_{n+1}(z^{-1} u\sigma_n - z^{-1} u\sigma_n^{-1}) =z^{-1}(a-a^{-1}) t_{n}(u)= \delta^H t_{n}(u) .$$
This implies that the only free parameter is $t_1(1)$ and that any trace necessarily satisfies the inclusion property with $d = \delta^H$. Hence there exists, up to normalization (i.e. the value $t_1(1)$), only one Markov trace on the tower of Hecke algebras.

For the choice of normalization $t_1(1)=\delta^H$, the corresponding Markov trace is multiplicative, let us denote it by $t_n^H$ as it is the one recovering the HOMFLY-PT polynomial. If $L$ is a link with writhe $w(L)= n_+ - n_-$ such that $L$ is the closure $\widehat{\bu}$ of a $n$-strands braid $\bu$, and if $\overline{\pi}$ is the projection from the group algebra of the $n$-strands braid group to $H_n$, then:
\begin{equation}
 P(L)(a,q) = a^{-w(L)} t_n^H(\overline{\pi} ( \bu))
\label{HOMFLYtH}
\end{equation}
where $P(L)(a,q)$ is the 2-variable HOMFLY--PT polynomial of the oriented link $L$. This polynomial is uniquely determined by its normalization, ie. its value on the trivial knot
$$P\left( \begin{array}{c}  \includegraphics[height=0.55cm]{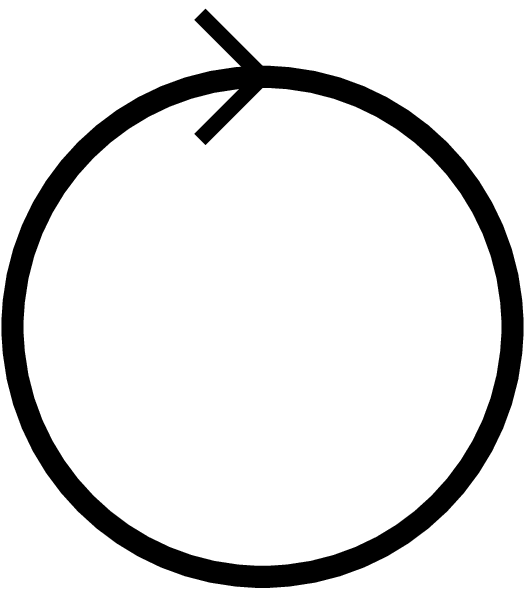} 
\end{array} \right)=\frac{a-a^{-1}}{q-q^{-1}},$$
and the following skein relation:
$$a P\left( \begin{array}{c}   \includegraphics{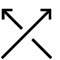} \end{array} 
\right)-a^{-1}P \left( \begin{array}{c}   \includegraphics{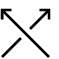} 
\end{array} \right) = (q-q^{-1})P\left( \begin{array}{c}   
\includegraphics[height=0.55cm]{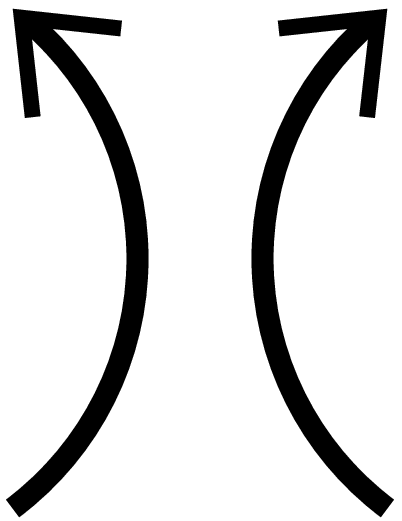} \end{array} \right).$$

\subsection{Markov traces and Kauffman polynomial}\label{subsec-Kauffman}

All the results of this section are well-known. The only aim here is to reprove, using the inductive basis $\mathbf{b}^{BMW}_{n}$ constructed previously, the unicity, up to normalization, of Markov traces satisfying the inclusion property with $d=\delta$ on the tower of BMW algebras $\left( BMW_n \right)_{n \geq 1}$.

In the sequel, we will repeatedly use the fact that
\begin{eqnarray*}
z \delta & = & z + a - a^{-1} \\
z \delta a^{-1} & = & za^{-1} + 1 - a^{-2} \\
z \delta a & = & za + a^2 - 1.
\end{eqnarray*}

Let us start with a preliminary lemma. Let $(t_n)_{n\geq1}$ be a family of $R$-linear maps with $t_n : BMW_n\to R$.
\begin{lem}
 Assume that $(t_n)_{n\geq 1}$ satisfies the stabilization and inclusion (with $d=\delta$) properties. For $n\geq 2$, we have
\begin{multline}\label{lem-eq1}
t_{n+1}(us_{n-1}^kvs_n^lw)=t_{n+1}(us_{n-1}^lvs_n^kw)\\
=t_{n+1}(us_{n}^kvs_{n-1}^lw)=t_{n+1}(us_{n}^lvs_{n-1}^lw),
\end{multline}
\begin{equation}\label{lem-eq2}
t_{n+1}(xs_n^{\epsilon}s_{n-1}^{-\epsilon}s_n^{\epsilon}y)= \epsilon z t_{n+1}(xs_n^{2\epsilon}y)+a^{\epsilon} t_n(xs_{n-1}^{2\epsilon}y)+ \epsilon z t_n(xs_{n-1}^{-2\epsilon}y),
\end{equation}
\begin{equation}\label{lem-eq3}
t_{n+1}(us_n^{\epsilon}s_{n-1}^{-\epsilon}s_n^{\epsilon} v)=a^{-\epsilon} t_n(us_{n-1}^{2\epsilon}v),
\end{equation}
where $k,l\in\mathbb{Z}$, $\epsilon= \pm 1$, $u,v,w\in BMW_{n-1}$ and $x,y\in BMW_n$.
\end{lem}
\begin{proof}
Note that for any $N\in\mathbb{Z}$, we have $s_i^N=\alpha_N+\beta_N s_i+\gamma_N s_i^{-1}$ with $\alpha_N,\beta_N,\gamma_N\in R$ independent of $i$. Then it is immediate to obtain \eqref{lem-eq1} by applying the stabilization and the inclusion properties.

Formulas \eqref{lem-eq2} are straightforward short calculations using Relations \eqref{BMW15} if $\epsilon = 1$ and \eqref{BMW16} if $\epsilon = -1$ and $\delta z=a-a^{-1}+z$.

To obtain \eqref{lem-eq3}, note first that $t_{n+1}(us_n^{2\epsilon}v)=t_{n+1}(u s_{n-1}^{2\epsilon}v)=\delta t_n(us_{n-1}^{2\epsilon}v)$ using \eqref{lem-eq1}. To conclude, it is enough to note that $t_n(us_{n-1}^{2\epsilon}v=t_n(us_{n-1}^{-2\epsilon}v)$ which follows easily from Relations \eqref{BMW9} and \eqref{BMW11}.
\end{proof}

Now we are ready to prove the main technical proposition.
\begin{prop}\label{stab+inc=trace}
Stabilization and inclusion (with $d=\delta$) properties imply the trace property.
\end{prop}

\begin{proof}
The proof is by induction on $n$. The trace property is satisfied for $n=1$ and we suppose it holds for $t_k$ with $k\leq n$.

We only need to check that the trace property $t_{n+1}(uv) = t_{n+1}(vu)$ is satisfied for  $u \in  \mathbf{b}^{BMW}_{n+1}$ and  $v=s_i$ for $i=1, \dots, n$. The property will then be satisfied in general since it follows that it is true for an element of the form $v = s_i^k$ for all $k \in \Z$ (immediate if $k \geq 0$ and using Relation \eqref{BMW9} if $k < 0$) and hence also true for any $v=s_{i_1}^{\epsilon_1}\dots s_{i_k}^{\epsilon_k}$with $\epsilon_j = \pm 1$ because we can now move the generators $s_{i_k}, \dots,s_{i_1}$ and their inverses one by one from right to left.

The case $i \leq n-1$ is straightforward. Since $u \in  \mathbf{b}^{BMW}_{n+1}$, there exist $u_1, u_2 \in BMW_n$ and $\epsilon \in \{ 0, \pm 1\}$ such that $u = u_1 s_n^{\epsilon} u_2$. It follows from the induction hypothesis that, if $\epsilon = 0$,
\begin{multline*}
t_{n+1}(uv)  =  t_{n+1}(u_1 u_2 s_i) = \delta t_{n}(u_1 u_2 s_i) = \delta t_{n}(s_i u_1 u_2 ) \\
= t_{n+1}(s_i u_1 u_2 ) = t_{n+1}(vu)
\end{multline*}
and, if $\epsilon = \pm 1$,
\begin{multline*}
t_{n+1}(uv)  =  t_{n+1}\left(u_1 s_n^{\pm 1} u_2 s_i\right) = a^{\pm 1} t_{n}(u_1 u_2 s_i) = a^{\pm 1} t_{n}(s_i u_1 u_2 )\\
 = t_{n+1}(s_i u_1 s_n^{\pm 1} u_2 ) = t_{n+1}(vu).
\end{multline*}

The case where $u \in  \mathbf{b}^{BMW}_{n}$ and  $v=s_n$ is treated similarly:
\begin{equation*}
t_{n+1}(uv)  =  t_{n+1}(u s_n) = a t_{n}(u) = t_{n+1}(s_n u) = t_{n+1}(vu).
\end{equation*}

So it remains to treat the cases where $v = s_n$ and $u$ is of one of the six following possible forms:
\begin{itemize}
\item[(1)] $ u= t x_{n,i} $ with $t \in \mathbf{b}^{BMW}_{n-1}$ and $ i = 1, \dots ,n$,
\item[(2)] $ u=  y_{i,n} t $ with $t \in \mathbf{b}^{BMW}_{n-1}$ and $ i = 1, \dots ,n$,
\item[(3)] $ u= t x_{n-1,j} x_{n,i} $ with $t \in \mathbf{b}^{BMW}_{n-1}$, $ j = 1, \dots ,n-1$ and $ i = 1, \dots ,n$,
\item[(4)] $ u=  y_{i,n} t x_{n-1,j}$ with $t \in \mathbf{b}^{BMW}_{n-1}$, $ j = 1, \dots ,n-1$ and $ i = 1, \dots ,n$,
\item[(5)] $ u= y_{j,n-1} t x_{n,i} $ with $t \in \mathbf{b}^{BMW}_{n-1}$, $ j = 1, \dots ,n-1$ and $ i = 1, \dots ,n$,
\item[(6)] $ u=  y_{i,n} y_{j,n-1} t$ with $t \in \mathbf{b}^{BMW}_{n-1}$, $ j = 1, \dots ,n-1$ and $ i = 1, \dots ,n$.
\end{itemize}

$\bullet$ Case $(1)$: if $i=n$, the proof is trivial since
$$ us_n = t s_n^{-1}s_n = s_n t s_n^{-1} = s_n u.$$

While if $i<n$, one has on one hand
\begin{multline*}
t_{n+1}(us_n ) = t_{n+1}(t x_{n,i}s_n ) = t_{n+1}\left(t s_n^{-1} s_{n-1}^{-1}x_{n-2,i} s_n \right) 
\\ = t_{n+1}\left(t s_n^{-1} s_{n-1}^{-1} s_n x_{n-2,i}\right) =  t_{n+1}\left(t s_{n-1} s_n^{-1} s_{n-1}^{-1} x_{n-2,i}\right) \\
= a^{-1} t_{n}\left(t s_{n-1} s_{n-1}^{-1} x_{n-2,i}\right) =a^{-1} \delta t_{n-1}(t x_{n-2,i})
\end{multline*}
and on the other hand
\begin{multline*}
t_{n+1}(s_n u) = t_{n+1}(s_n t x_{n,i} ) = t_{n+1}\left(t s_n s_n^{-1} s_{n-1}^{-1} x_{n-2,i} \right) \\
 = \delta t_{n}\left(t s_{n-1}^{-1} x_{n-2,i}\right) =\delta  a^{-1} t_{n-1}(t x_{n-2,i}).
\end{multline*}

$\bullet$ Case $(2)$: if $i=n$, the proof is trivial since
$$ us_n = s_n t s_n = s_n^2 t = s_n u.$$

If $i<n$, one can use Relation \eqref{lem-eq1} to compute
\begin{multline*}
t_{n+1}(us_n ) = t_{n+1}(y_{i,n} t s_n ) = t_{n+1}\left(y_{i,n-2}s_{n-1}t s_n^2 \right) = t_{n+1}\left(y_{i,n-2}s_{n-1}^2t s_n \right) 
\end{multline*}
and 
\begin{multline*}
t_{n+1}(s_n u ) = t_{n+1}(s_n y_{i,n} t ) = t_{n+1}(y_{i,n-2}s_n s_{n-1} s_n t ) \\
= t_{n+1}(y_{i,n-2} s_{n-1} s_n s_{n-1} t ) = a t_{n}\left(y_{i,n-2} s_{n-1}^2 t \right)= t_{n+1}\left(y_{i,n-2} s_{n-1}^2 t s_n\right).
\end{multline*}

$\bullet$ Case $(3)$: if $i=n$, we have 
$$ t_{n+1} (u s_n) = t_{n+1} \left( t x_{n-1,j} s_n^{-1} s_n\right) = \delta t_{n} \left( t s_{n-1}^{-1}x_{n-2,j} \right) = \delta a^{-1} t_{n-1} ( t x_{n-2,j} ) $$
and 
$$ t_{n+1} (s_n u) = t_{n+1} \left(s_n t x_{n-1,j} s_n^{-1}\right) = \delta t_{n} \left( t s_n s_{n-1}^{-1} s_n^{-1} x_{n-2,j} \right) = \delta a^{-1} t_{n-1} ( t x_{n-2,j} ). $$

If $i<n$, one can see that
\begin{multline*}
t_{n+1} (u s_n) = t_{n+1} ( t x_{n-1,j} x_{n,i} s_n) = t_{n+1} \left( t x_{n-1,j} s_n^{-1} s_{n-1}^{-1} s_n x_{n-2,i} \right) \\
= t_{n+1} \left( t x_{n-1,j} s_{n-1} s_n^{-1} s_{n-1}^{-1} x_{n-2,i} \right) = a^{-1} t_{n} \left( t x_{n-1,j} s_{n-1} s_{n-1}^{-1} x_{n-2,i} \right) \\
= a^{-1} t_{n} \left( t s_{n-1}^{-1}  x_{n-2,j} x_{n-2,i} \right) = a^{-2} t_{n-1} ( t  x_{n-2,j} x_{n-2,i} )
\end{multline*}
and
\begin{multline*}
t_{n+1} (s_n u ) = t_{n+1} (s_n t x_{n-1,j} x_{n,i} ) = t_{n+1} \left( ts_n s_{n-1}^{-1} x_{n-2,j} s_n^{-1} x_{n-1,i} \right) \\
= t_{n+1} \left( ts_n s_{n-1}^{-1} s_n^{-1} x_{n-2,j} x_{n-1,i} \right) = t_{n+1} \left( t s_{n-1}^{-1} s_n^{-1} s_{n-1} x_{n-2,j} x_{n-1,i} \right) \\
= a^{-1} t_{n} \left( t s_{n-1}^{-1} s_{n-1} x_{n-2,j}  x_{n-1,i} \right) = a^{-1} t_{n} \left( t  x_{n-2,j} s_{n-1}^{-1} x_{n-2,i} \right) \\
= a^{-2} t_{n-1} ( t  x_{n-2,j} x_{n-2,i} ).
\end{multline*}

$\bullet$ Case $(4)$: if $i=n$, we obtain, using Relation \eqref{lem-eq3}, that
\begin{multline*}
 t_{n+1} (u s_n) = t_{n+1} ( s_n t x_{n-1,j}  s_n) = t_{n+1} \left(  t s_n s_{n-1}^{-1} s_n x_{n-2,j}\right)\\
 =a^{-1}t_n\left(  t s_{n-1}^{2} x_{n-2,j}\right)=t_{n+1}\left( s_n^{-1} t s_{n-1}^{2} x_{n-2,j}\right)\ .
\end{multline*}
Using Relation \eqref{lem-eq1}, we can see that also
\[ t_{n+1} (s_n u ) = t_{n+1} \left( s_n^2 t s_{n-1}^{-1}x_{n-2,j}\right) =t_{n+1} \left( s_n^{-1} t s_{n-1}^{2}x_{n-2,j}\right).\]

If $i<n$, we will proceed in two steps. First one can use Relation \eqref{lem-eq2} plus Relations \eqref{BMW9} and \eqref{BMW11} to check that
\begin{multline*}
t_{n+1} (u s_n) = t_{n+1} (y_{i,n} t x_{n-1,j}  s_n) =  t_{n+1} \left(y_{i,n-1} t s_n s_{n-1}^{-1} s_n x_{n-2,j} \right) \\
= -z  t_{n+1} \left(y_{i,n-1} t s_n^2 x_{n-2,j} \right) + at_{n} (y_{i,n-1} t s_{n-1}^2 x_{n-2,j} ) + z t_{n+1} \left(y_{i,n-1} t s_{n-1}^{-2} x_{n-2,j} \right) \\
=-z \big[(1-za^{-1}) \delta + (a^{-1} + z)a  -a^{-2} \big] t_{n} (y_{i,n-1} t  x_{n-2,j} ) \\
+a \big[ (1-za^{-1}) t_{n} (y_{i,n-1} t  x_{n-2,j} ) + (a^{-1} + z) t_{n} (y_{i,n-1} t s_{n-1} x_{n-2,j} ) \bigr. \\
\bigl.  -a^{-1}  t_{n} (y_{i,n-1} t  x_{n-1,j} ) \big] + z \big[ (1+za) t_{n} (y_{i,n-1} t  x_{n-2,j} ) \bigr. \\
\bigl.+ (a - z) t_{n} (y_{i,n-1} t  x_{n-1,j})   -a t_{n} (y_{i,n-1} t s_{n-1} x_{n-2,j} )    \big] \\
= \big[-za \left( (1-za^{-1}) \delta + (a^{-1} + z)a  -a^{-2} \right) + a^2(1-za^{-1}) \bigr. \\
\bigl. +za (1+za) \big] t_{n-1} (y_{i,n-2} t  x_{n-2,j} ) + \left( a(a^{-1} + z) -az \right)t_{n} (y_{i,n-1} t s_{n-1} x_{n-2,j} ) \\
+ \bigl( -1+z(a - z) \bigr) t_{n} (y_{i,n-1} t  x_{n-1,j} ) = \left( 1 + z^2 -za \right) t_{n-1} (y_{i,n-2} t  x_{n-2,j} ) \\
+ t_{n} (y_{i,n-1} t s_{n-1} x_{n-2,j} ) + \left( -1+za - z^2 \right) t_{n} (y_{i,n-1} t  x_{n-1,j} )
\end{multline*}
while, simply using Relation \eqref{BMW9}, we obtain that 
\begin{multline*}
t_{n+1} (s_n u ) = t_{n+1} (s_n y_{i,n} t x_{n-1,j} ) = t_{n+1} ( y_{i,n-2} s_n s_{n-1} s_n t x_{n-1,j}  ) \\
= t_{n+1} ( y_{i,n-2} s_{n-1} s_n s_{n-1} t x_{n-1,j}  ) =a  t_{n} \left( y_{i,n-2} s_{n-1}^2 t x_{n-1,j}  \right) \\
=a\left(1-za^{-1}\right)  t_{n} ( y_{i,n-2}  t x_{n-1,j}  ) + a\left(a^{-1} + z\right) t_{n} ( y_{i,n-1}  t x_{n-1,j}  ) \\
- t_{n} ( y_{i,n-2} s_{n-1}^{-1} t x_{n-1,j}  ) =\left(1-za^{-1}\right)  t_{n-1} ( y_{i,n-2}  t x_{n-2,j}  ) \\
+ (1 + za) t_{n} ( y_{i,n-1}  t x_{n-1,j}  ) - t_{n} \left( y_{i,n-2} s_{n-1}^{-1} t x_{n-1,j}  \right). 
\end{multline*}

Secondly, we will use the induction hypothesis and apply the trace property in order to simplify the four terms $t_{n-1} ( y_{i,n-2}  t x_{n-2,j}  )$, $t_{n} ( y_{i,n-1}  t x_{n-1,j}  )$, $t_{n} (y_{i,n-1} t s_{n-1} x_{n-2,j} )$ and $t_{n} \left( y_{i,n-2} s_{n-1}^{-1} t x_{n-1,j}  \right)$. We have to treat three different subcases depending on $i$ and $j$.

If $i=j$, we have immediately that
\begin{eqnarray*}
t_{n-1} ( y_{i,n-2}  t x_{n-2,i}  ) & = & t_{n-1} ( t  ) \\
t_{n} ( y_{i,n-1}  t x_{n-1,j}  ) & = & \delta t_{n-1} ( t  ) \\
t_{n} (y_{i,n-1} t s_{n-1} x_{n-2,i} ) & = &  t_{n} \left( s_{n-1}^2 t  \right) \\
& = & \left( \delta +za - za^{-1} \right) t_{n-1} ( t  ) \quad \text{by \eqref{BMW9}} \\
t_{n} \left( y_{i,n-2} s_{n-1}^{-1} t x_{n-1,i}  \right)  & = &  t_{n} \left( s_{n-1}^{-2} t  \right) \\
& = & \left( \delta +za - za^{-1} \right) t_{n-1} ( t  ) \quad \text{by \eqref{BMW11}.} 
\end{eqnarray*}
Hence it follows that
\begin{multline*}
t_{n+1} (u s_n) = \left( 1 + z^2 -za \right) t_{n-1} ( t  ) + \left( \delta +za - za^{-1} \right) t_{n-1} ( t  ) \\
 + \left( -1+za - z^2 \right) \delta t_{n-1} ( t  ) = a^2 \delta t_{n-1} ( t  )
\end{multline*}
and 
\begin{multline*}
t_{n+1} (s_n u ) = \left(1-za^{-1}\right)  t_{n-1} ( t  ) + (1 + za) \delta t_{n-1} ( t  ) \\
 - \left( \delta +za - za^{-1} \right) t_{n-1} ( t  ) = a^2 \delta t_{n-1} ( t  ). 
\end{multline*}

If $i < j$, one can check that 
\begin{eqnarray*}
t_{n-1} ( y_{i,n-2}  t x_{n-2,j}  ) & = & t_{n-1} ( x_{n-2,j} y_{i,n-2}  t   ) \\
& = & t_{n-1} ( y_{i,n-2} x_{n-3,j-1} t   ) \\
t_{n} ( y_{i,n-1}  t x_{n-1,j}  ) & = & t_{n} (x_{n-1,j}  y_{i,n-1}  t  ) \\
& = & t_{n} (  y_{i,n-1} x_{n-2,j-1} t  ) \\
& = & a t_{n-1} (  y_{i,n-2} x_{n-2,j-1} t  ) \\
& = & a t_{n-1} (  y_{i,j-2}t  ) \\
t_{n} (y_{i,n-1} t s_{n-1} x_{n-2,j} ) & = & t_{n} (s_{n-1} x_{n-2,j} y_{i,n-1} t  ) \\
& = & t_{n} (y_{i,n-1} s_{n-2} x_{n-3,j-1} t  ) \\
& = & a t_{n-1} \left(y_{i,n-3} s_{n-2}^2 x_{n-3,j-1} t  \right) \\
& = & (a-z) t_{n-1} (  y_{i,j-2}t  ) \\
& & + (1 + za)t_{n-1} ( y_{i,n-2} x_{n-3,j-1} t   ) \\
& &  - t_{n-1} ( y_{i,n-3} x_{n-2,j-1} t   ) \quad \text{by \eqref{BMW9}} \\
\end{eqnarray*}
and, using Relation \eqref{lem-eq2} plus Relations \eqref{BMW9} and \eqref{BMW11}, that
\begin{multline*}
t_{n} \left( y_{i,n-2} s_{n-1}^{-1} t x_{n-1,j}  \right) = t_{n} \left( y_{i,n-3} s_{n-1}^{-1} s_{n-2} s_{n-1}^{-1} x_{n-3,j-1}t  \right) \\
=  a^{-1} t_{n-1} \left( y_{i,n-3}s_{n-2}^{-2}  x_{n-3,j-1}t  \right) + z  t_{n} \left( y_{i,n-3}s_{n-1}^{-2}  x_{n-3,j-1}t  \right) \\
-z t_{n-1} \left( y_{i,n-3} s_{n-2}^{2} x_{n-3,j-1}t  \right)\\
= a^{-1} t_{n-1} \left( y_{i,n-3}s_{n-2}^{-2}  x_{n-3,j-1}t  \right) + z \big[  (1+za) \delta +(a-z)a^{-1}- a^{2} \big] t_{n-1} ( y_{i,j-2} t  ) \\
-z t_{n-1} \left( y_{i,n-3} s_{n-2}^{2} x_{n-3,j-1}t  \right) \\
= a^{-1} \big[ (1+za)t_{n-1} ( y_{i,j-2} t  ) +(a-z) t_{n-1} ( y_{i,n-3}  x_{n-2,j-1}t  ) \bigr.\\
\bigl.  -a t_{n-1} ( y_{i,n-2}  x_{n-3,j-1}t  ) \big] + z \big[ \delta +z\left(a-a^{-1}\right) \big] t_{n-1} ( y_{i,j-2} t  ) \\
-z \big[(1-za^{-1})t_{n-1} ( y_{i,j-2} t  ) +(a^{-1} + z)t_{n-1} ( y_{i,n-2}  x_{n-3,j-1}t  ) \bigr.  \\
\bigl. - a^{-1} t_{n-1} ( y_{i,n-3}  x_{n-2,j-1}t  ) \big] = \left(z^2a + a+z\right) t_{n-1} (  y_{i,j-2}t  )\\
- \left(1 + za^{-1} + z^2\right)t_{n-1} ( y_{i,n-2} x_{n-3,j-1} t   ) + t_{n-1} ( y_{i,n-3} x_{n-2,j-1} t   ).
\end{multline*}
Then it follows directly that
\begin{multline*}
t_{n+1} (u s_n) =  \left( 1 + z^2 -za \right) t_{n-1} ( y_{i,n-2} x_{n-3,j-1} t   ) \\
+ (a-z) t_{n-1} (  y_{i,j-2}t  ) + (1 + za)t_{n-1} ( y_{i,n-2} x_{n-3,j-1} t   ) - t_{n-1} ( y_{i,n-3} x_{n-2,j-1} t   )\\
 + \left( -1+za - z^2 \right) a t_{n-1} (  y_{i,j-2}t  ) = \left(za^2- z^2a - z \right) t_{n-1} (  y_{i,j-2}t  ) \\
 + \left( 2 + z^2 \right) t_{n-1} ( y_{i,n-2} x_{n-3,j-1} t   ) - t_{n-1} ( y_{i,n-3} x_{n-2,j-1} t   )
\end{multline*}
and
\begin{multline*}
t_{n+1} (s_n u )  =\left(1-za^{-1}\right)  t_{n-1} ( y_{i,n-2} x_{n-3,j-1} t   ) \\
+ (1 + za) a t_{n-1} (  y_{i,j-2}t  ) - \left(z^2a + a+z\right) t_{n-1} (  y_{i,j-2}t  ) \\
 + \left(1 + za^{-1} + z^2\right)t_{n-1} ( y_{i,n-2} x_{n-3,j-1} t   ) - t_{n-1} ( y_{i,n-3} x_{n-2,j-1} t   )\\
= \left(za^2- z^2a - z \right) t_{n-1} (  y_{i,j-2}t  ) \\
+ \left( 2 + z^2 \right) t_{n-1} ( y_{i,n-2} x_{n-3,j-1} t   ) - t_{n-1} ( y_{i,n-3} x_{n-2,j-1} t   ).
\end{multline*}

If $i > j$, one can check that 
\begin{eqnarray*}
t_{n-1} ( y_{i,n-2}  t x_{n-2,j}  ) & = & t_{n-1} ( x_{n-2,j} y_{i,n-2}  t   ) \\
& = & t_{n-1} ( y_{i-1,n-3} x_{n-2,j} t   ) \\
t_{n} ( y_{i,n-1}  t x_{n-1,j}  ) & = & t_{n} (x_{n-1,j}  y_{i,n-1}  t  ) \\
& = & t_{n} (  y_{i-1,n-2} x_{n-1,j} t  ) \\
& = & a^{-1} t_{n-1} (  y_{i-1,n-2} x_{n-2,j} t  ) \\
& = & a^{-1} t_{n-1} (  x_{i-2,j}t  ) \\
t_{n} \left(y_{i,n-2} s_{n-1}^{-1} t  x_{n-1,j} \right) & = & t_{n} \left(x_{n-1,j} y_{i,n-2} s_{n-1}^{-1} t   \right) \\
& = & t_{n} \left( y_{i-1,n-3} s_{n-2}^{-1} s_{n-1}^{-1} s_{n-2}^{-1} x_{n-3,j} t   \right)  \\
& = & a^{-1} t_{n-1} \left( y_{i-1,n-3} s_{n-2}^{-2} x_{n-3,j} t   \right)  \\
& = & \left(a^{-1} + z\right) t_{n-1} (   x_{i-2,j}t )\\
&  & + \left(1 - za^{-1}\right)t_{n-1} ( y_{i-1,n-3} x_{n-2,j} t   )  \\
& &  - t_{n-1} ( y_{i-1,n-2} x_{n-3,j} t   ) \quad \text{by \eqref{BMW11}} \\
\end{eqnarray*}
and, using Relation \eqref{lem-eq2} plus Relations \eqref{BMW9} and \eqref{BMW11}, that
\begin{multline*}
t_{n} ( y_{i,n-1} t s_{n-1} x_{n-2,j}  ) = t_{n} \left( y_{i-1,n-3} s_{n-1} s_{n-2}^{-1} s_{n-1} x_{n-3,j}t \right ) \\
=-z  t_{n} \left(y_{i-1,n-3} s_{n-1}^2 x_{i-2,j}t  \right) + a t_{n-1} \left( y_{i-1,n-3} s_{n-2}^2 x_{n-3,j}t  \right) \\
 +z t_{n-1} \left( y_{i-1,n-3} s_{n-2}^{-2} x_{n-3,j}t  \right) \\
= -z \left[ \left(1-za^{-1}\right) \delta + \left(a^{-1} + z\right) a - a^{-2} \right] t_{n-1} (   x_{i-2,j}t ) \\
+ a t_{n-1} \left( y_{i-1,n-3} s_{n-2}^2 x_{n-3,j}t  \right) +z t_{n-1} \left( y_{i-1,n-3} s_{n-2}^{-2} x_{n-3,j}t  \right) \\
= -z \left[ \delta + z \left(a - a^{-1}\right) \right] t_{n-1} (   x_{i-2,j}t ) + a \left[ \left(1-za^{-1}\right) t_{n-1} (   x_{i-2,j}t ) \right. \\
 \left. + \left(a^{-1} + z\right) t_{n-1} ( y_{i-1,n-2} x_{n-3,j} t   ) - a^{-1} t_{n-1} ( y_{i-1,n-3} x_{n-2,j} t   ) \right]\\
+ z \left[ (1+za) t_{n-1} (   x_{i-2,j}t ) + (a-z) t_{n-1} ( y_{i-1,n-3} x_{n-2,j} t   ) \right. \\
\left. -a t_{n-1} ( y_{i-1,n-2} x_{n-3,j} t   ) \right] = \left(z^2a^{-1} + a^{-1} -z\right) t_{n-1} (   x_{i-2,j}t ) \\
+ \left(-1 +az -z^2\right) t_{n-1} ( y_{i-1,n-3} x_{n-2,j} t   ) + t_{n-1} ( y_{i-1,n-2} x_{n-3,j} t   ).
\end{multline*}

From this we can immediately deduce that 
\begin{multline*}
t_{n+1} (u s_n) =  \left( 1 + z^2 -za \right) t_{n-1} ( y_{i-1,n-3} x_{n-2,j} t   ) \\
+\left(z^2a^{-1} + a^{-1} -z\right) t_{n-1} (   x_{i-2,j}t )+ \left(-1 +az -z^2\right) t_{n-1} ( y_{i-1,n-3} x_{n-2,j} t   ) \\
 + t_{n-1} ( y_{i-1,n-2} x_{n-3,j} t   ) + \left( -1+za - z^2 \right) a^{-1} t_{n-1} (  x_{i-2,j}t  )\\
 = t_{n-1} ( y_{i-1,n-2} x_{n-3,j} t   )
\end{multline*}
and
\begin{multline*}
t_{n+1} (s_n u )  =\left(1-za^{-1}\right)  t_{n-1} ( y_{i-1,n-3} x_{n-2,j} t   )\\
+ (1 + za) a^{-1} t_{n-1} (   x_{i-2,j}t ) - \left(a^{-1}+z\right) t_{n-1} (   x_{i-2,j}t ) \\
 - \left(1 - za^{-1} \right)t_{n-1} ( y_{i-1,n-3} x_{n-2,j} t   ) + t_{n-1} ( y_{i-1,n-2} x_{n-3,j} t   ) \\
= t_{n-1} ( y_{i-1,n-2} x_{n-3,j} t   ).
\end{multline*}

$\bullet$ Case $(5)$: if $i=n$, we have 
$$ t_{n+1} (u s_n) = t_{n+1} \left( y_{j,n-1} t  s_n^{-1} s_n\right) = \delta t_{n} ( y_{j,n-2} s_{n-1} t  ) = \delta a t_{n-1} (y_{j,n-2} t  ) $$
and 
\begin{multline*}
 t_{n+1} (s_n u) = t_{n+1} \left(s_n y_{j,n-1} t s_n^{-1}\right) =  t_{n+1} \left( y_{j,n-2}  s_n s_{n-1} s_n^{-1} t \right) \\
 =  t_{n+1} \left( y_{j,n-2} s_{n-1}^{-1}  s_n s_{n-1} t\right) =  a t_{n} \left( y_{j,n-2} s_{n-1}^{-1} s_{n-1} t \right) = a \delta  t_{n-1} (y_{j,n-2}  t). 
\end{multline*}

If $i<n$, one can see that
\begin{multline*}
t_{n+1} (u s_n) = t_{n+1} (y_{j,n-1} t x_{n,i}  s_n) = t_{n+1} \left(y_{j,n-1} t s_n^{-1} s_{n-1}^{-1} s_n x_{n-2,i}  \right) \\
= t_{n+1} \left(y_{j,n-1} t s_{n-1} s_n^{-1} s_{n-1}^{-1} x_{n-2,i}  \right) = a^{-1} t_{n} \left(y_{j,n-1} t s_{n-1} s_{n-1}^{-1} x_{n-2,i}  \right)  \\
= a^{-1} t_{n} (y_{j,n-2} s_{n-1} t x_{n-2,i}  ) = t_{n-1} ( y_{j,n-2} t x_{n-2,i} )
\end{multline*}
and
\begin{multline*}
t_{n+1} (s_n u ) = t_{n+1} (s_n y_{j,n-1} t x_{n,i} ) = t_{n+1} \left( y_{j,n-2} s_n s_{n-1}  s_n^{-1} t x_{n-1,i}  \right) \\
= t_{n+1} \left(y_{j,n-2} s_{n-1}^{-1} s_n s_{n-1} t x_{n-1,i}  \right) = a t_{n} \left(y_{j,n-2} s_{n-1}^{-1} s_{n-1} t x_{n-1,i} \right)  \\
= a t_{n} \left( y_{j,n-2} t s_{n-1}^{-1} x_{n-2,i}   \right) = t_{n-1} ( y_{j,n-2} t x_{n-2,i} ).
\end{multline*}

$\bullet$ Case $(6)$: if $i=n$, we obtain
\begin{multline*}
 t_{n+1} (u s_n) = t_{n+1} (s_n y_{j,n-1} t s_n) =  t_{n+1} ( y_{j,n-2} s_n s_{n-1} s_n t  ) \\
= t_{n+1} ( y_{j,n-2} s_{n-1} s_n s_{n-1}  t  ) = a t_{n} \left(y_{j,n-2} s_{n-1}^2 t  \right)=  t_{n+1} \left(s_n y_{j,n-2} s_{n-1}^2 t  \right)
\end{multline*}
and, using Relation \eqref{lem-eq1}, 
\[
 t_{n+1} (s_n u ) = t_{n+1} (s_n^2 y_{j,n-2}s_{n-1} t) = t_{n+1} \left(s_n y_{j,n-2} s_{n-1}^2 t  \right)\ .\]

If $i<n$, we have immediately that
\begin{multline*}
 t_{n+1} (u s_n) = t_{n+1} ( y_{i,n} y_{j,n-1} t s_n) =  t_{n+1} (y_{i,n-1} y_{j,n-2} s_n s_{n-1} s_n t  ) \\
=  t_{n+1} (y_{i,n-1} y_{j,n-2} s_{n-1} s_n s_{n-1}  t  ) \\
= \begin{cases}
a t_{n} \left(y_{i,n-2} s_{n-1}^{3} t  \right) & \text{if $j=n-1$,}\\
\text{or} & \\
\begin{array}{l}
a t_{n} \left(y_{i,n-2} y_{j,n-3} s_{n-1} s_{n-2} s_{n-1}^{2} t  \right) \\
 = a t_{n} \left(y_{i,n-2} y_{j,n-3} s_{n-2}^{2} s_{n-1} s_{n-2}  t  \right)\\
 = a^{2} t_{n-1} \left(y_{i,n-2} y_{j,n-3} s_{n-2}^{3} t  \right)
\end{array}
& \text{if $j<n-1$,}
\end{cases} 
\end{multline*}
and
\begin{multline*}
 t_{n+1} (s_n u ) = t_{n+1} (s_n y_{i,n} y_{j,n-1} t ) =  t_{n+1} (y_{i,n-2} s_n  s_{n-1}s_n y_{j,n-2} s_{n-1} t  ) \\
=  t_{n+1} (y_{i,n-2}  s_{n-1} s_n s_{n-1}  y_{j,n-2} s_{n-1} t  ) \\
= \begin{cases}
a t_{n} \left(y_{i,n-2} s_{n-1}^{3} t  \right) & \text{if $j=n-1$,}\\
\text{or} & \\
\begin{array}{l}
a t_{n} \left(y_{i,n-2} y_{j,n-3} s_{n-1}^2 s_{n-2} s_{n-1} t  \right) \\
 = a t_{n} \left(y_{i,n-2} y_{j,n-3} s_{n-2} s_{n-1} s_{n-2}^{2} t  \right)\\
 = a^{2} t_{n-1} \left(y_{i,n-2} y_{j,n-3} s_{n-2}^{3} t  \right)
\end{array}
& \text{if $j<n-1$.}
\end{cases} 
\end{multline*}
\end{proof}

\begin{thm}\label{reconstr Kauffman}
There exists, up to normalization, a unique Markov trace satisfying the inclusion property with $d=\delta$ on the tower of BMW algebras.
\end{thm}

\begin{proof}
Unicity: it is obvious from the inductive construction of the bases $\mathbf{b}^{BMW}_{n}$ for $n\geq 1$ that, under stabilization and inclusion, the only free parameter is $t_1(1)$.

Existence: we check that the family $(t_n)_{n \geq1}$ of linear maps defined inductively as follows is indeed a Markov trace. The value $t_1(1)$ is just a free parameter and, if $r \in BMW_{n+1}$ has the following expression in the basis $\mathbf{b}^{BMW}_{n+1}$:
$$r= \sum\limits_{i=1}^k b_i u_i s_n^{\epsilon_i} v_i + \sum\limits_{j=1}^l c_j w_j $$
where $u_i, v_i, w_j \in BMW_n$ and $b_i, c_j \in R$ for all $i,j$, then we define
$$t_{n+1}(r)= \sum\limits_{i=1}^k a^{\epsilon_i}b_i t_n(u_i v_i) + \delta \sum\limits_{j=1}^l c_j t_n(w_j).$$
We only need to prove that $(t_n)_{n \geq1}$ satisfies the stabilization and inclusion properties since these imply the trace property by Proposition \ref{stab+inc=trace}. The inclusion is obvious: if $r \in BMW_{n}$ then $b_i = 0$ for all $ i=1, \cdots, k$ and it remains $t_{n+1}(r)= \delta \sum\limits_{j=1}^l c_j t_n(w_j) = \delta t_n(r)$. To check the stabilization property, one needs to verify that for all $u,v \in BMW_n$, one has $t_{n+1}\left(u s_n^{\pm1}v\right) = a^{\pm1} t_n(uv)$. Let us consider the linear map
\begin{equation} \label{top closure map}
Cl_{n+1} : BMW_{n+1} \xrightarrow{\quad \simeq \quad} KT_{n+1} \xrightarrow{\quad \quad} KT_n \xrightarrow{\quad \simeq \quad} BMW_{n}
\end{equation}
where the middle arrow is the topological closure of the $(n+1)$st strand of the tangle. We will use the following lemma.

\begin{lem}
For all $r\in BMW_{n+1}$, we have $t_{n+1} (r) = t_n\left(Cl_{n+1}(r)\right)$.
\end{lem}

\begin{proof}
It sufficient to prove it for all $r\in \mathbf{b}^{BMW}_{n+1}$. For these basis elements, the equality follows directly from the skein relations \eqref{skeinrel} holding in $KT_{n+1}$.
\end{proof}

Note that this lemma only holds because we supposed that the trace satisfies the inclusion property. From this we can conclude
$$t_{n+1}\left(u s_n^{\pm1}v\right) = t_n \left( Cl_{n+1}\left(u s_n^{\pm1}v\right)\right) = t_n\left(a^{\pm1}uv\right) =  a^{\pm1} t_n(uv).$$
\end{proof}

\begin{rems} Let us make now a few observations about this trace, how it relates to polynomial invariants of links and also why we allowed ourselves to suppose the inclusion property beforehand.
\begin{itemize}
	\item The trace $t_n^K$ on the tower of BMW algebras constructed in the theorem and for the choice of normalization $t_1^K(1) = \delta$ gives the two-variable Kauffman polynomial. As observed before, this choice of normalization implies that this trace is multiplicative. The Kauffman polynomial is defined, for any link $L$, by:
	$$ F(L)(a,q) = a^{-w(L)} \left< L \right>(a,q)$$
where $w(L)= n_+ - n_-$ is the writhe of $L$ and $ \left< L \right>$ its Kauffman bracket defined via the skein relations \eqref{skeinrel}. If $L$ equals to $\widehat{\bu}$, the closure of a $n$-strands braid $\bu$, and if $\pi$ is the projection from the group algebra of the $n$-strands braid group to $BMW_n$ then:
$$ F(L)(a,q) = a^{-w(L)} t_n^K(\pi ( \bu)).$$
\item The projection $\overline{\pi}$ from the group algebra of the $n$-strands braid group to the Hecke algebra actually factorizes through the BMW algebra as $\overline{\pi} = q_E \circ \pi $. Thus the Markov trace $t_n^H$ on the tower of Hecke algebras can be viewed as a Markov trace on the tower of BMW algebras just by setting $t_n^H(u) = t_n^H(q_E(u))$ for all $u \in BMW_n$.
\item For a more general classification of Markov traces on BMW algebras, see \cite{MaWaBMW}. In this paper, Marin and Wagner prove in full generality (without assuming the inclusion property) that, for generic $a$ and $q$, a Markov trace on the tower of BMW algebras only depends on the two parameters $t_1(1)$ and $t_2(e_1)$. Actually this general result follows easily from Equation \eqref{trtr1'} of Section \ref{trtrBMW} of the present paper as, for any $u \in BMW_{n}$, one has
\begin{eqnarray*}
t_{n+1}(u) & = & z^{-1} t_{n+1}(u s_{n}) - z^{-1} t_{n+1}(u s_{n}^{-1}) + t_{n+1}(u e_{n}) \\
& = & z^{-1}(a- a^{-1}) t_n(u)+ t_{n}(Cl_{n}(u) e_{n-1}).
\end{eqnarray*} 
The two traces $t_n^H$ and $t_n^K$, which are linearly independent, hence form a basis of the space of Markov traces on the tower of BMW algebras. This explains why we could restrict ourselves to looking for a trace on the BMW algebras which satisfies in particular the inclusion property for $d=\delta$. The fact that $\delta \neq \delta^H$ also insures that this trace is linearly independent from $t_n^H$, since if $t_n^K = \alpha t_n^H$ then the inclusion parameters would have to be equal.
\end{itemize}
\end{rems}

\section{Transverse Markov traces}

For a kind introduction to transverse links and transverse braids, we suggest the survey of Etnyre \cite{EtnySur} and references therein.\\

\subsection{Definition}

A transverse Markov trace on the tower of algebras $(A_n)_{n\geq1}$ over $R$ is a family $(\tau_n)_{n\geq 1}$ of $R$--linear maps $\tau_n : A_n \rightarrow R$ satisfying the following conditions:
\begin{itemize}
\item trace property: $\tau_n(uv) = \tau_n(vu)$,
\item positive stabilization property: $\tau_{n+1} \left( \iota_n(u) g_n \iota_n(v) \right) = a \tau_n(uv)$
\end{itemize}
for all $u,v \in A_n$ and $n\geq 1$.

Note that, unlike for a classical Markov trace, we do not suppose negative stabilization. We also relaxed a priori another hypothesis, which was the inclusion property. 

From now on, let $(\tau_n)_{n\geq 1}$ be a transverse Markov trace.

\subsection{Transverse Markov traces on Hecke algebras}

We will use one of the inductive bases of the Hecke algebra $H_{n+1}$ recalled in Section \ref{subsec-bases}. It is defined recursively by $\mathbf{b}^{H}_{1}= \{ 1 \}$ and (we omit the inclusion map $\iota_n$):
\[ \mathbf{b}^{H}_{n+1}=\{ b, b\sigma_n\dots\sigma_i | b \in \mathbf{b}^{H}_{n}, i = 1, \dots, n \}.\]
Another description of the basis elements of $H_{n+1}$ is as follows:
\[ \mathbf{b}^{H}_{n+1}=\{1,u\s_k\s_{k-1}\dots\s_i |\ k=1,\dots,n,\ i=1,\dots,k,\ u\in \mathbf{b}_k^H \}\]
\subsubsection{Classification}

\begin{thm}\label{classification}
A transverse Markov trace $(\tau_n)_{n \geq1}$  on $(H_n)_{n\geq1}$ is uniquely determined by the family of parameters $\{\tau_n(1)\}_{n \geq1}\subset R$, and is given recursively on the basis $\mathbf{b}^{H}_{n+1}$ by:
\begin{equation}\label{rhon-base}
\tau_{n+1}(u\s_k\s_{k-1}\dots\s_i)=a \tau_{n}(u\s_{k-1}\dots\s_i),
\end{equation}
for $n\geq 1$, $k=1,\dots,n$, $i=1,\dots,k$ and $u\in \mathbf{b}_k^H$.
\end{thm}

Note that we need to give this recursive formula for all $k=1, \dots, n$ and not only for $k=n$ as we do not suppose that the trace satisfies the inclusion property.

\begin{proof} First we note that the defining formula (\ref{rhon-base}) implies the following more general formula:
\begin{equation}\label{rhon}
\tau_{n+1}(u\s_k v)=a \tau_{n}(uv) \quad \text{for $n\geq 1$, $k=1,\dots,n$ and $u,v\in H_{k}$.}
\end{equation}
Indeed, it is enough to prove (\ref{rhon}) for $u,v$ basis elements in $\mathbf{b}_k^H$. So let $v=v_1\s_{k-1}\dots \s_i$, where $v_1\in H_{k-1}$ and $i\in\{1,\dots,k\}$ (with the convention that $v=v_1$ if $i=k$). We have indeed that, using (\ref{rhon-base}),
\[\tau_{n+1}(u\s_k v)=\tau_{n+1}(uv_1\s_k\s_{k-1}\dots \s_i)=a \tau_{n}(u v_1\s_{k-1}\dots \s_i)=a \tau_{n}(uv) .\]

Then we proceed to show that any transverse Markov trace must satisfy (\ref{rhon}), and hence the defining formula (\ref{rhon-base}). Let $n\geq 1$, $k\in\{1,\dots,n\}$ and $u,v\in H_{k}$. We have:
\[\s_{n}\dots \s_{k+1}u\s_k v \s_{k+1}^{-1}\dots\s_{n}^{-1}=u \s_k^{-1}\dots\s_{n-1}^{-1}\s_{n}\s_{n-1}\dots\s_kv,\]
where we used the fact that $u,v$ commute with $\s_{n},\dots, \s_{k+1}$ and the braid relations. Therefore, if $\tau_n$ is a transverse Markov trace on $H_n$, we must have:
\[\begin{array}{ll}
\tau_{n+1}(u\s_k v) & =\tau_{n+1}\left(\s_{n}\dots \s_{k+1}u\s_k v \s_{k+1}^{-1}\dots\s_{n}^{-1}\right)\\
 & =\tau_{n+1}\left(u \s_k^{-1}\dots\s_{n-1}^{-1}\s_{n}\s_{n-1}\dots\s_k v\right) \\
 & =a \tau_{n}\left(u \s_k^{-1}\dots\s_{n-1}^{-1}\s_{n-1}\dots\s_k v\right)\\
 & =a\tau_{n}(uv) ,\end{array}
 \]
where we used the trace property in the first equality, and the positive stabilization property in the third one. Therefore any transverse Markov trace must satisfy this recursive formula. Furthermore, this recursive formula and the data of $\{\tau_n(1)\}_{n \geq1}\subset R$ are obviously sufficient to fully determine this transverse Markov trace.
 
Now it remains to show that, for any choice of parameters $\{\tau_n(1)\}_{n \geq1}\subset R$, the family of maps $(\tau_n)_{n \geq1}$ given in the theorem is indeed a well-defined transverse Markov trace. 
 
First, for $n\geq1$, by taking $k=n$ in (\ref{rhon}), we have at once the positive stabilization property. Thus only the trace property needs to be checked. 

We will proceed by induction on $n$. For $n=1,2$, the trace property is automatically satisfied since $H_1$ and $H_2$ are commutative. So let $n\geq 3$ and let $u,v\in H_{n+1}$. It is enough to take for $u$ an element of the basis $\mathbf{b}^H_{n+1}$ and for $v$ a generator of $H_{n+1}$. If $u=1$ there is nothing to do, so let 
\[u=u_1\s_k u_2, \quad\text{where $k\in\{1,\dots,n\}$ and $u_1,u_2\in H_k$,}\]
and let $v=\s_{i}$, where $i\in\{1,\dots,n\}$. We have to show that $\tau_{n+1}(uv)=\tau_{n+1}(vu)$.

$\bullet$ Case $(1)$: if $i > k$, we have $u\in H_i$ and therefore
\[\tau_{n+1}(uv)=\tau_{n+1}(u \s_i )=a\tau_{n}(u)=\tau_{n+1}( \s_i u) = \tau_{n+1}(vu)\]
where we used (\ref{rhon}) for $\s_i$ twice.

$\bullet$ Case $(2)$: if $i < k$, then $\s_i \in H_k$. We have, now using (\ref{rhon}) for $\s_k$:
\begin{multline*}\tau_{n+1}(uv)=\tau_{n+1}(u_1\s_k u_2\s_i )=a\tau_{n}(u_1 u_2\s_i) \\
=a\tau_{n}(\s_iu_1 u_2)=\tau_{n+1}(\s_i u_1\s_k u_2 )= \tau_{n+1}(vu).\end{multline*}
where the induction hypothesis, namely the trace property for $\tau_{n}$, yielded the desired result.

$\bullet$ Case $(3)$: if $i = k$, we need to describe more precisely $u_1$ and $u_2$. In general, $u_1,u_2$ can be taken of the following form:
\[\left\{\begin{array}{ll}
u_2=\s_{k-1}\dots \s_j \qquad & \text{where $j\in\{1,\dots,k\}$,}\\
u_1=w_1\s_{l} w_2 & \text{where $l\in\{0,1,\dots,k-1\}$ and $w_1,w_2\in H_{l}$ .}
\end{array}\right.\]
By convention, $j=k$ means $u_2=1$ and $l=0$ means $u_1=1$.

First, we transform $vu=\s_k w_1\s_{l} w_2\s_k u_2$. If $l<k-1$, $\s_k$ commutes with $u_1=w_1\s_{l} w_2$. If $l=k-1$, we move the first $\s_k$ through $w_1$ and the second through $w_2$ and apply the braid relation $\s_k\s_{k-1}\s_k=\s_{k-1}\s_k\s_{k-1}$. We obtain:
\[vu
=\left\{\begin{array}{ll}
u_1\s_k^2u_2 & \text{if $l<k-1$,}\\
w_1\s_{k-1}\s_k\s_{k-1}w_2u_2 \qquad & \text{if $l=k-1$.}
\end{array}\right.\]
Next we apply $\tau_{n+1}$. We use the definition and the quadratic relation for the generators to obtain:
\begin{equation}\label{rho-xy}
\tau_{n+1}(vu)=\left\{\begin{array}{ll}
\tau_{n+1}(u_1u_2)+za\tau_{n}(u_1u_2) & \text{if $l<k-1$,}\\
a\tau_{n}(w_1w_2u_2)+az\,\tau_{n}(w_1\s_{k-1}w_2u_2) \quad & \text{if $l=k-1$.}
\end{array}\right.
\end{equation}

Then, we transform $uv=u_1\s_k\s_{k-1}\dots \s_j\s_k$. If $j<k$, we use the braid relations and we obtain:
\[uv
=\left\{\begin{array}{ll}
u_1\s_k^2 & \text{if $j=k$,}\\
u_1\s_{k-1}\s_k\s_{k-1}\dots\s_j \qquad & \text{if $j<k$.}
\end{array}\right.\]
We apply $\tau_{n+1}$, use the definition and the quadratic relation for the generators and get:
\begin{equation}\label{rho-yx}
\tau_{n+1}(uv)=\left\{\begin{array}{ll}
\tau_{n+1}(u_1)+za\tau_n(u_1) & \text{if $j=k$,}\\
a\tau_{n}(u_1\s_{k-2}\dots\s_j)+a z\tau_{n}(u_1\s_{k-1}\dots\s_j) \quad & \text{if $j<k$.}
\end{array}\right.
\end{equation}
To conclude we must compare (\ref{rho-xy}) and (\ref{rho-yx}). There are four cases:
\begin{itemize}
\item ($j=k$ and $l<k-1$) Here $u_2=1$ and the right hand sides of (\ref{rho-xy}) and (\ref{rho-yx}) coincide.
\item ($j=k$ and $l=k-1$) Here also $u_2=1$ and the result follows from $\tau_{n+1}(u_1)=a\tau_{n}(w_1w_2)$.
\item ($j<k$ and $l<k-1$) Since $u_1\in H_{l+1}$, we have
$$\tau_{n+1}(u_1u_2)=\tau_{n+1}(u_1\s_{k-1}\s_{k-2}\dots\s_j)=a\tau_{n}(u_1\s_{k-2}\dots\s_j) $$
so that the right hand sides of (\ref{rho-xy}) and (\ref{rho-yx}) coincide as well.
\item ($j<k$ and $l=k-1$) In the right hand sides of (\ref{rho-xy}) and (\ref{rho-yx}), the terms with coefficient $az$ coincide. For the other terms, we have on one hand
\[
a\tau_{n}(w_1w_2u_2)=a\tau_{n}(w_1w_2\s_{k-1}\dots \s_j)=a^2\tau_{n-1}(w_1w_2\s_{k-2}\dots\s_j) ,\]
and on the other hand
\[a\tau_{n}(u_1\s_{k-2}\dots\s_j)=a\tau_{n}(w_1\s_{k-1}w_2\s_{k-2}\dots\s_j)=a^2\tau_{n-1}(w_1w_2\s_{k-2}\dots\s_j).\]
\end{itemize}
The proof of the theorem is complete.
\end{proof}

\begin{rem}\label{procedure} Let $n\geq 1$ and $u\in H_n$. Given a transverse Markov trace $(\tau_n)_{n\geq1}$, the recursive Formula \eqref{rhon} allows to express explicitly $\tau_n(u)$ in terms of the parameters $\tau_1(1),\dots,\tau_n(1)$. For example,
\[\tau_3(\s_2\s_1\s_2)=\tau_3(\s_1\s_2\s_1)=a\tau_2\left(\s_1^2\right)=a\tau_2(1+z\s_1)=a\tau_2(1)+a^2z\tau_1(1).\]
\end{rem}

\subsubsection{Arbitrary and basic transverse Markov traces}\label{arbbastrtr}

Together with the classification obtained above, we have an explicit description of an arbitrary transverse Markov trace on $(H_n)_{n\geq1}$. In particular, given a family of parameters $\{\alpha_n\}_{n\geq1}\subset R$, a unique transverse Markov trace $(\tau_n)_{n\geq1}$ is determined by fixing:
\[\tau_n(1)=\alpha_n,\quad \forall n\geq1.\]
We call the set $\{\alpha_n\}_{n\geq1}$ the parameters of the transverse Markov trace. It will sometimes be useful to see the parameters $\{\alpha_n\}_{n\geq1}$ as generic (that is, as indeterminate).

For $k\in\Z_{>0}$, we denote by $\left(\tau_n^{(k)}\right)_{n\geq1}$ the transverse Markov trace with parameters $\{\delta_{n,k}\}_{n\geq1}$ (that is, $\alpha_k=1$ and $\alpha_{l}=0$ if $l\neq k$).
In terms of these transverse Markov traces, an arbitrary one $\left(\tau_n\right)_{n\geq1}$ with parameters $\{\alpha_n\}_{n\geq1}$ is given by:
\begin{equation}\label{arbitrary}
\left(\tau_n\right)_{n\geq1}=\sum_{k\geq 1}\alpha_k\left(\tau_n^{(k)}\right)_{n\geq1} .
\end{equation}
Indeed, according to Theorem \ref{classification}, the two sides of \eqref{arbitrary} coincide since they both give a transverse Markov trace on $(H_n)_{n\geq1}$, and they coincide on 1 for each $n\geq1$.

We call the traces $\left(\tau_n^{(k)}\right)_{n\geq1}$ for all $k>0$ the basic transverse Markov traces. An arbitrary transverse Markov trace can be expressed in terms of these linearly independent basic traces, hence they form a basis of the space of transverse Markov traces on the tower of Hecke algebras $(H_n)_{n\geq 1}$.

\begin{rem} Given an arbitrary transverse Markov trace on $(H_n)_{n\geq1}$, it is a classical Markov trace if and only if its parameters $\{\alpha_n\}_{n\geq1}$ satisfy
\[\alpha_n=\left(\delta^H\right)^{n-1}\alpha_1 .\]

Recall that the Markov trace $\left(t_n^H\right)_{n\geq1}$ giving rise to the HOMFLY--PT polynomial is obtained by further setting $\alpha_1=\delta^H$. In other words, the expansion of this trace in terms of the basic transverse Markov traces is:
\begin{equation}\label{tau-basic}
t_n^H=\sum_{1\leq k\leq n}\left(\delta^H\right)^k\tau_n^{(k)},\qquad \forall n\geq1 .
\end{equation}
\end{rem}

\begin{example} Let $n=2$ and consider $\s_1^{-1}\in H_2$. Recall that $\s_1^{-1}=\s_1-z$. For an arbitrary transverse Markov trace $\left(\tau_n\right)_{n\geq1}$ with parameters $\{\alpha_n\}_{n\geq1}$, we have:
\[\tau_2\left(\s_1^{-1}\right)=\tau_2(\s_1)-z\tau_2(1)=a\tau_1(1)-z\tau_2(1)=a\alpha_1-z\alpha_2 .\]
Thus, we have $\tau_2^{(1)}\left(\s_1^{-1}\right)=a$ and $\tau_2^{(2)}\left(\s_1^{-1}\right)=-z$. 

We note that $a^{-1} \cdot 1\in H_1$ and $\sigma_1^{-1}\in H_2$ are distinguished by the transverse Markov traces. Indeed, we have 
$$a^{-1}\tau_1(1)=a^{-1}\alpha_1 \neq a\alpha_1-z\alpha_2=\tau_2\left(\s_1^{-1}\right) ,$$ 
for generic parameters $\alpha_1,\alpha_2$. The coincidence occurs precisely when $\alpha_2=z^{-1}\left(a-a^{-1}\right)\alpha_1$, which is the condition for being a classical Markov trace.
\end{example}

\subsubsection{Multiplicativity}

 Let $\{\alpha_n\}_{n\geq1}$ be generic parameters. For any $R$--linear combination $c(\alpha_1,\alpha_2,\dots)$ in the parameters $\{\alpha_n\}_{n\geq1}$, we denote by $pf(c)$ the linear combination given by
\[pf(c)(\alpha_1,\alpha_2,\dots)=c(\alpha_2,\alpha_3,\dots) ,\]
that is, $pf(c)$ is obtained from $c$ by pushing forward by $1$ the indices of the parameters.

Let $(\tau_n)_{n\geq1}$ be an arbitrary transverse Markov trace on $(H_n)_{n\geq1}$ with parameters $\{\alpha_n\}_{n\geq1}$. Note that, for $n\geq 1$ and $u\in H_n$, $\tau_n(u)$ is a linear combination of $\alpha_1,\dots,\alpha_n$. Thus, for any $n\geq 1$ and  any $u\in H_n$, $pf\left(\tau_n(u)\right)$ is well--defined.

\begin{lem}\label{lem-pf}
The transverse Markov trace $(\tau_n)_{n\geq1}$ with parameters $\{\alpha_n\}_{n\geq1}$ is given by the following formulas:
\begin{equation}\label{rhon-pf}
\begin{array}{rcll}
\tau_1(1)&=&\alpha_1, & \\
\tau_n(u\s_{n-1}v)&=&a\tau_{n-1}(uv)\quad & \text{for any $n\geq2$ and $u,v\in H_{n-1}$,}\\
\tau_n(u)&=&pf\left(\tau_{n-1}(u)\right)\quad & \text{for any $n\geq2$ and $u\in H_{n-1}$.}
\end{array}
\end{equation}
\end{lem}

\begin{proof} The two first equations in (\ref{rhon-pf}) are prescribed by the definition of $(\tau_n)_{n\geq1}$. Let $n\geq2$. It is enough to prove the third equation for $u$ an element of the standard basis of $H_{n-1}$. If $u=1$, we indeed have:
$$\tau_n(1)=\alpha_n=pf(\alpha_{n-1})=pf\left(\tau_{n-1}(1)\right) .$$
Now let $u$ be a basis element of $H_{n-1}$ different from 1. For $n=2$ we have nothing more to prove. We proceed then by induction on $n$. Let $n\geq3$. As $u\neq 1$, we have $u=u_1\s_k u_2$, where $k\in\{1,\dots,n-2\}$ and $u_1,u_2\in H_{k}$. We have:
\begin{multline*}
\tau_n(u)=\tau_n(u_1\s_ku_2)=a\tau_{n-1}(u_1u_2)=a \cdot pf\left(\tau_{n-2}(u_1u_2)\right)\\
=pf\left(\tau_{n-1}(u_1\s_k u_2)\right)=pf\left(\tau_{n-1}(u)\right),
\end{multline*}
where we used the defining formula \eqref{rhon} and the induction hypothesis.
\end{proof}

\begin{prop}\label{prop-mult}
Let $(\tau_n)_{n\geq1}$ be a transverse Markov trace with parameters $\{\alpha_n\}_{n\geq1}$. It is multiplicative if and only if
\begin{equation}\label{mult-cond}
\alpha_n=\alpha_1^n, \qquad \text{for all $n\geq 2$.}
\end{equation}
\end{prop}
\begin{proof} 
Let $n_1,n_2\in\Z_{>0}$, $u_1\in H_{n_1}$ and $u_2\in H_{n_2}$. By taking $(n_1,n_2)=(n-1,1)$ and $u_1=u_2=1$, we obtain that the multiplicativity implies $\alpha_n=\alpha_{n-1}\alpha_1$ for any $n\geq 2$, and therefore the condition (\ref{mult-cond}) is necessary.

Besides, we have in general that $\tau_{n_i}(u_i)$ is a linear combination of the parameters $\alpha_1,\dots,\alpha_{n_i}$ ($i=1,2$). Thus let,
\begin{equation}
\tau_{n_1}(u_1)=\lambda_1\alpha_1+\cdots+\lambda_{n_1}\alpha_{n_1}\quad \text{and}\quad \tau_{n_2}(u_2)=\mu_1\alpha_1+\cdots+\mu_{n_2}\alpha_{n_2} .
\label{multu2}
\end{equation}
To express $\tau_{n_1+n_2}(u_1\otimes u_2)$, we start by applying the transverse Markov condition to reduce $u_2$. Indeed Equation \eqref{multu2} just means that the algorithm applied to compute $\tau_{n_2}(u_2)$ allows to reduce it to $\mu_1 \tau_1(1) + \cdots + \mu_{n_2} \tau_{n_2}(1)$, which in turns straightforwardly implies that the same algorithm leads to:
\[\tau_{n_1+n_2}(u_1\otimes u_2)=\mu_1\tau_{n_1+1}(u_1\otimes 1)+\mu_2\tau_{n_1+2}(u_1 \otimes 1)+\cdots+\mu_{n_2}\tau_{n}(u_1 \otimes 1) .\]
Then we use Lemma \ref{lem-pf}, namely that $\tau_{n_1+k}(u_1)=pf^k\left(\tau_{n_1}(u_1)\right)$, to obtain
\begin{multline*}
\tau_{n_1+n_2}(u_1\otimes u_2)= \mu_1 pf\left(\tau_{n_1}(u_1)\right)+\cdots+\mu_{n_2}pf^{n_2}\left(\tau_{n_1}(u_1)\right) = \\
\mu_1 (\lambda_1 \alpha_2 + \cdots + \lambda_{n_1} \alpha_{n_1+1})+\cdots+\mu_{n_2}(\lambda_1 \alpha_{n_2+1} + \cdots + \lambda_{n_1} \alpha_{n}) = \sum_{i=1}^{n_1} \sum_{j=1}^{n_2} \lambda_i\mu_j\alpha_{i+j} ,
\end{multline*}
and this must be equal, if $\{\tau_n\}_{n\geq1}$ is multiplicative, to $\displaystyle \sum_{i=1}^{n_1} \sum_{j=1}^{n_2}\lambda_i\mu_j\alpha_{i}\alpha_j$. Therefore, for all $i,j \geq 1$, we must have $\alpha_{i+j} = \alpha_i \alpha_j$ and hence the condition \eqref{mult-cond} is also sufficient.
\end{proof} 

\begin{rem} Note that we do not loose any information by restricting to multiplicative transverse Markov traces. Indeed, let $\alpha_1$ be generic, then the family of parameters $\{\alpha_1^n\}_{n\geq 1}$ gives rise to a multiplicative Markov trace $(\tau_n)_{n\geq1}$ and for any $u \in H_n$, $\tau_n(u)$ is simply a degree $n$ polynomial in $\alpha_1$. We can then obtain the values of the basic transverse Markov trace $\tau^{(k)}_n(u)$ for any $k\geq 1$ as the coefficient of $\alpha_1^k$ in $\tau_n(u)$.
\end{rem}

\subsubsection{Inequalities} 

As explained in the introduction, the Morton-Franks-Williams inequality relating the degree in $a$ of the HOMFLY--PT polynomial and the self-linking number can be seen as directly related to the classification of transverse Markov traces on the Hecke algebras. This is why we include here a proof derived from the classification of the previous subsection.\\

Recall that $B_n$ is the braid group on $n$ strands, and let us denote by $\bs_1,\dots,\bs_{n-1}$ the standard braid generators in $B_n$. They project onto the generators $\s_1, \dots, \s_{n-1}$ of the Hecke algebra $H_n$ via $\overline{\pi}$ and onto the generators $s_1, \dots, s_{n-1}$ of the BMW algebra $BMW_n$ via $\pi$. For an arbitrary braid $\bu\in B_n$, we write $\bu$ as a word in the generators and their inverses $\bu=\bs_{i_1}^{\epsilon_1}\dots\bs_{i_N}^{\epsilon_N}$. Let us also recall here two numerical invariants of braids, which are respectively the braid index and the writhe:
\[i(\bu)=n\qquad \text{and}\qquad w(\bu)=\epsilon_1+\dots\epsilon_N .\]
We note that $w(\bu)$ is indeed independent of the expression of $\bu$ as a word in the generators as the braid relations are homogeneous.

We consider generic parameters $\{\alpha_n\}_{n\geq1}$. For a non-zero linear combination in the parameters $\{\alpha_n\}_{n\geq1}$ with coefficients in $R$, we define
\[e(P_1\alpha_1+P_2\alpha_2+\dots)=\text{min}\{-\deg_{a}(P_k)-k\ |\ k\geq1\ \text{such that}\ P_k\neq 0\} ,\]
where $\deg_{a}(P)$, for a Laurent polynomial $P$ in $a$, is the maximal power of $a$ appearing in $P$. By convention, we set $e(0)=-\infty$.

Let $(\tau_n)_{n\geq1}$ be the transverse Markov trace on $\{H_n\}_{n\geq1}$ with family of parameters $\{\alpha_n\}_{n\geq1}$. For $\bu\in B_n$, we set
\begin{equation}\label{dbeta}
d(\bu)=e\left(\tau_n(u)\right) ,
\end{equation}
where for simplicity we denote by $u$ the image $\overline{\pi}(\bu)$ of $\bu$ in $H_n$.

\begin{rems}\label{Ineqrks} Let us observe the following facts.

\begin{enumerate}
\item For $k\geq1$, let us recall the basic transverse Markov traces $\left(\tau_n^{(k)}\right)_{n\geq1}$ and define, for any $\bu\in B_n$, 
\[d^{(k)}(\bu)=
-\deg_{a}\left(\tau_n^{(k)}(u)\right)-k,\qquad  \text{if $\tau_n^{(k)}(u)\neq0$.}\]
Then, by definition, we have
\[d(\bu)=\text{min}\{d^{(k)}(\bu)\ |\ k\geq1\ \text{such that}\ \tau_n^{(k)}(u)\neq 0\} .\]
\item Assume that the parameters $\{\alpha_n\}_{n\geq1}$ of the transverse Markov trace $(\tau_n)_{n\geq1}$ are specialized in $R$ in such a way that $\deg_{a}(\alpha_k)=k$ for every $k\geq1$. Then we have 
$$d^{(k)}(\bu)= -\deg_{a}\left(\tau_n^{(k)}(u)\right) - \deg_{a}(\alpha_k) = - \deg_{a}\left(\tau_n^{(k)}(u)\alpha_k \right) $$
and hence
\[d(\bu)\leq -\deg_{a}\left(\tau_n(u)\right) .\]
In particular, the classical trace $t_n^H$ satisfies this condition so that we have
\[d(\bu)\leq - \deg_{a}\left(t_n^H(u)\right) ,\]
where $\deg_{a}\left(t_n^H(u)\right)$ is then, by Equation \eqref{HOMFLYtH}, nothing but $\deg_{a}\left(P(\widehat{\bu})\right) + w(\bu) $, i.e. the minimal power of $a$ appearing in the HOMFLY--PT polynomial of $\widehat{\bu}$ plus its writhe $w(\widehat{\bu}) = w(\bu)$.
\end{enumerate}
\end{rems}

\begin{thm}
For every braid $\bu$, we have:
\begin{equation}\label{inequality1}
-i(\bu)\leq d(\bu) .
\end{equation}
\end{thm}
\begin{proof} Fix $n\geq1$ and let $\bu\in B_n$ with $\bu=\bs_{i_1}^{\epsilon_1}\dots\bs_{i_N}^{\epsilon_N}$
and $\epsilon_1,\dots,\epsilon_N=\pm1$. We will proceed by induction on $N$.

If $N=0$, then $\bu=1\in B_n$ and we have $-i(\bu)=-n$. On the other hand, we have $\tau_n(u)=\tau_n(1)=\alpha_n$. Thus, by the defining formula \eqref{dbeta}, we have $d(\bu)=-n$ as well.

So let $n\geq 2$ and $N>0$. We will first observe that it is sufficient to prove the Theorem in the case where all the $\epsilon_i$ are positive. Indeed, if we denote by $N_-$ the number of exponents among $\epsilon_1,\dots,\epsilon_N$ equal to $-1$, then, using the relation \eqref{H4} holding in the Hecke algebra $H_n$, we can rewrite $\s_{i_j}^{-1}$ as $\s_{i_j}-z$ and hence we have
\[u=\s_{i_1}^{\epsilon_1}\dots \s_{i_N}^{\epsilon_N}=\s_{i_1}\dots\s_{i_N}+\sum_{k=1}^{N_-}(-z)^k v_k,\]
where $v_k$ is the sum of all possible subwords of $\s_{i_1}\dots\s_{i_N}$ of length $N-k$ where we forget $k$ of the generators that appear in the expression of $\bu$ with negative exponent. We can apply the induction hypothesis to the $v_k$'s as their length is strictly smaller than $N$ and we have, for each $k=1,\dots,N_-$,
\[e\left(\tau_n( v_k)\right)\geq -n ,\]
and therefore
\[e\left(\tau_n\left(\sum_{k=1}^{N_-}(-z)^k v_k\right)\right)\geq -n .\]

Thus, it just remains to show that
\[d(\bs_{i_1}\dots\bs_{i_N}) = e\left(\tau_n(\s_{i_1}\dots\s_{i_N})\right)\geq -n \]
and hence we can restrict ourselves to proving the Theorem for elements of the positive braid monoid.

So now we assume that $\bu=\bs_{i_1}\dots \bs_{i_N}\in B_n$, with $N>0$, and we shall prove (\ref{inequality1}) in this case. We will distinguish two cases depending if the image of $\bu$ in the symmetric group through the projection $B_n \rightarrow S_n$, $\bs_i \mapsto (i,i+1)$, is a reduced element or not.

$\bullet$ Case $(1)$: assume that the image of $\bu$ in the symmetric group is a reduced element, and  let $k\in\{1,\dots,n-1\}$ be the largest index value such that $\bs_k$ appears in the word of $\bu$. Then, with the use of the braid relations, it is possible to rewrite $\bu$ as a word where $\bs_k$ appears only once. Therefore, we can assume that $\bs_k$ appears only once in the word of $\bu$. Now, let $\bu'$ be the word obtained from $\bs_{i_1}\dots \bs_{i_N}$ be removing $\bs_k$. Using the recursive formula (\ref{rhon}) for a transverse Markov trace, we have
\[\tau_n(u)=a\tau_{n-1}(u') .\]
Consequently, we consider the word $\bu'$ as an element of $B_{n-1}$ which is of length strictly smaller than $N$ so we use the induction hypothesis to get that
\[d(\bu)=d(\bu') - 1\geq-(n-1) -1=-n .\]

$\bullet$ Case $(2)$: assume that the image of $\bu$ in the symmetric group is not a reduced element. Again, using the braid relations, we can assume that, for some $k\in\{1,\dots,n-1\}$, two $\bs_k$ are consecutive in the word representing $\bu$. So let now
\[\bu=\bs_{i_1}\dots \bs_{i_l}\bs_k^2\bs_{i_{l+3}}\dots \bs_{i_N} ,\]
and using the quadratic relation \eqref{H3} in $H_n$, we obtain
\[u=\s_{i_1}\dots \s_{i_l}\s_{i_{l+3}}\dots \s_{i_N}+ z\s_{i_1}\dots \s_{i_l}\s_k\s_{i_{l+3}}\dots \s_{i_N} .\]
Now we use the induction hypothesis which asserts that
\[\left\{\begin{array}{l}
e\left(\tau_n(\s_{i_1}\dots \s_{i_l}\s_{i_{l+3}}\dots \s_{i_N})\right)\geq -n ,\\
e\left(\tau_n(\s_{i_1}\dots \s_{i_l}\s_k\s_{i_{l+3}}\dots \s_{i_N})\right)\geq -n .\end{array}\right.\]
This yields to
\[d(\bu)\geq -n=-i(\bu) ,\]
which concludes the proof.
\end{proof}

As a direct corollary of this theorem and of the second point of Remarks \ref{Ineqrks}, we get back the Francks-Morton-Williams inequality \cite{FrWi}, \cite{MorIne}. 

\begin{cor}
For every braid $\bu$, we have:
\begin{equation}\label{inequality}
w(\bu)-i(\bu)\leq - \deg_{a}\left(P(\widehat{\bu})\right) .
\end{equation}
\end{cor}

\subsection{Transverse Markov traces on BMW algebras}\label{trtrBMW}

Cubical quotients of the group algebra of the braid group could a priori carry interesting transverse Markov traces. We prove in this section that in some sense the BMW algebra is too small to have such traces and that all transverse Markov traces on it are determined by the ones on the Hecke algebras and by the two-variable Kauffman polynomial.\\

Let $(\tau_n)_{n \geq1}$ be a transverse Markov trace on $(BMW_n)_{n\geq1}$.

\begin{lem}
Let $n,n_1,n_2\geq 1$. For any $v_1 \in BMW_{n_1}$, $v_2 \in BMW_{n_2}$ and $v\in BMW_n$, the following holds:
\begin{equation}
\tau_{n_1 + n_2}\left(v_1 sh^{n_1}(v_2) s_{n_1} \right) =  a  \tau_{n_1+n_2 -1}\left(v_1 sh^{n_1 -1}(v_2)\right),
\label{trtr2B}
\end{equation}
\begin{equation}
\tau_{n+2}(v e_{n+1}) = \delta  \tau_{n+1}(v e_{n}),
\label{trtr4}
\end{equation}
\begin{equation}
\tau_{n+2}\left(v s_{n+1}^{-1}\right) = -a \delta  \tau_{n}(v) +(a +z \delta)  \tau_{n+1}(v) + \delta \tau_{n+1}\left(v s_{n}^{-1}\right) - z \tau_{n+2}(v).
\label{trtr5}
\end{equation}
\end{lem}

\begin{proof}
Relation \eqref{trtr2B} can be obtained using conjugation:
\begin{multline*}
\tau_{n_1 + n_2}\left(v_1 sh^{n_1}(v_2) s_{n_1}\right) = \tau_{n_1 + n_2}\left(v_1 x_{ n_1+n_2-1, n_1+1} sh^{n_1}(v_2) s_{n_1} y_{n_1+1, n_1+n_2-1}\right)  \\
= \tau_{n_1 + n_2}\left(v_1  y_{n_1, n_1+n_2-2} s_{n_1+n_2-1} x_{ n_1+n_2-2, n_1} sh^{n_1-1}(v_2)\right) = \\
 a \tau_{n_1 + n_2}\left(v_1  y_{n_1, n_1+n_2-2} x_{ n_1+n_2-2, n_1} sh^{n_1-1}(v_2)\right) = 
a  \tau_{n_1+n_2 -1}\left(v_1 sh^{n_1 -1}(v_2)\right).
\end{multline*}

Relation \eqref{trtr4} is a consequence of \eqref{BMW4'}, \eqref{BMW12} and \eqref{trtr2B}:
\begin{multline*}
\tau_{n+2}(v e_{n+1}) = a^{-1} \tau_{n+2}(v e_{n+1} s_n e_{n+1}) = a^{-1} \tau_{n+2}\left(v e_{n+1}^2 s_n \right) = \\ 
a^{-1} \delta \tau_{n+2}(v e_{n+1} s_n ) = \delta  \tau_{n+1}(v e_{n}).
\end{multline*}

Relation \eqref{trtr5} is obtained straightforwardly using twice Relation \eqref{defei} and Relation \eqref{trtr4}.
\end{proof}

Recall the definition of the map $Cl_{n+1}$ given in (\ref{top closure map}).
\begin{lem}
For any $u \in BMW_{n+1}$, the following relation is satisfied:
\begin{equation}
\tau_{n+2}(u e_{n+1}) =  \tau_{n+1}\left(Cl_{n+1}(u) e_{n}\right),
\label{trtr1'}
\end{equation}
\end{lem}

\begin{proof}
One first use Relation \eqref{BMW4'} and then Relation \eqref{trtr2B} to get:
\begin{multline*}
\tau_{n+2}(u e_{n+1}) = a^{-1} \tau_{n+2}(u e_{n+1} s_n e_{n+1}) = a^{-1} \tau_{n+2}(e_{n+1} u e_{n+1} s_n) \\ = a^{-1} \tau_{n+2}\left(Cl_{n+1}(u) e_{n+1} s_n \right) =  \tau_{n+1}\left(Cl_{n+1}(u) e_{n}\right)
\end{multline*}
since $e_{n+1}ue_{n+1}=Cl_{n+1}(u)e_{n+1}$ and $Cl_{n+1}(u) \in BMW_n$.
\end{proof}

\begin{prop}\label{propinddesc'}
 Let $2\leq k\leq n$ and $r\in BMW_{k+1}$. There exists $u,v\in BMW_{k}$ such that
 \[\tau_{n+1}(r)=\tau_{n+1}(u)+\tau_n(v) .\]
\end{prop}
\begin{proof}
It is enough to check that this is satisfied for $r$ a basis element of $BMW_{k+1}$. Then, by the inductive definition of the basis, there exists $b \in \mathbf{b}^{BMW}_{k}$ such that $r$ is of one of the following $3$ forms: $r =b$ or $r = y_{i,k} b$ or $r= b x_{k,i}$, with $i\in\{1,\dots,k\}$.
If $r=b$ then the proposition is obvious. Otherwise, let us consider the element $b'=sh^{n+1-(k+1)}(b)$ and use the fact that $\tau_{n+1}(x)=\tau_{n+1}(sh^{n+1-j}(x))$ for $j\in\{1,\dots,n+1\}$ and $x\in BMW_j$. If $r = y_{i,k} b$, we obtain
	\begin{multline*}\tau_{n+1}(r)=\tau_{n+1}(y_{i+n-k,n}b')=\tau_{n+1}(y_{i+n-k,n-1}s_nb')\\
	=a\tau_{n}(y_{i+n-k,n-1}b')=a\tau_n(y_{i,k-1}b) .
	\end{multline*}

If $r = b x_{k,i}$ we apply Relation \eqref{defei} and Relation \eqref{trtr1'} and we get the following
	\begin{multline*}
	\tau_{n+1}(r) = \tau_{n+1}(b'x_{n,i+n-k})=  \tau_{n+1}\left(b' s_{n}^{-1} x_{n-1,i+n-k}\right) \\
	= \tau_{n+1}(b' s_{n} x_{n-1,i+n-k}) + z \tau_{n+1}(b' e_{n} x_{n-1,i+n-k}) - z \tau_{n+1}(b' x_{n-1,i+n-k}) \\
	= a \tau_{n}(b' x_{n-1,i+n-k}) + z \tau_{n}\left(Cl_{n}(x_{n-1,i+n-k}b') e_{n-1} \right) - z \tau_{n+1}(b' x_{n-1,i+n-k}) \\
	= a \tau_{n}(b x_{k-1,i}) + z \tau_{n}\left(Cl_{k}(x_{k-1,i}b) e_{k-1} \right) - z \tau_{n+1}(b x_{k-1,i}) ,
	\end{multline*}
where we used that $Cl_n\circ sh^{n-k}=sh^{n-k}\circ Cl_k$.
\end{proof}

Let us denote by $\alpha_i$ the value $\tau_i(1)$ and by $\beta$ the value $\tau_2(e_1)$. 

\begin{thm}\label{thmtrtrBMW1}
For any $r \in BMW_{n}$, there exist $b, c_i \in R$, with $i = 1, \cdots, n$, such that
$$\tau_{n} (r) = b \beta + \sum_{i=1}^{n} c_i \alpha_i.$$
\end{thm}

\begin{proof}
The theorem follows by decreasing induction from Proposition \ref{propinddesc'}. The only terms which remain to be treated at the end of the induction, are of the form $\tau_k(r)$, with $r\in BMW_2$ and $k\geq 2$, and $\tau_1(1)$. We have $\tau_k(1)=\alpha_k$, $\tau_k(s_1)=a\tau_{k-1}(1)=a\alpha_{k-1}$ and, as for $\tau_k\left(s_1^{-1}\right)$, we can rewrite this term, if $k>2$, according to Relation \eqref{trtr5} as follows:
\begin{multline*}
 \tau_k\left(s_1^{-1}\right) = \tau_k\left(s_{k-1}^{-1}\right)  = -a \delta  \tau_{k-2}(1) +(a +z \delta)  \tau_{k-1}(1) + \delta \tau_{k-1}\left( s_{k-2}^{-1}\right) - z \tau_{k}(1) \\ 
=  -a \delta  \tau_{k-2}(1) +(a +z \delta)  \tau_{k-1}(1) + \delta \tau_{k-1}\left( s_{1}^{-1}\right) - z \tau_{k}(1) .
\end{multline*}
So we are left with the only problematic term $\tau_2\left(s_1^{-1}\right)$, to which one can apply Relation \eqref{BMW7} to get:
$$\tau_2\left(s_1^{-1}\right)= \tau_2(s_1) + z\tau_2(e_1) - z\tau_2(1) = a \tau_1(1) + z\tau_2(e_1) - z\tau_2(1).$$
This proves the theorem.
\end{proof}

This Theorem provides an algorithm to compute the transverse trace of any element of $BMW_n$ in terms of the $\alpha_i$'s and $\beta$. Whatever the transverse trace one starts with, if one applies this algorithm, the linear combination obtained at the end has the same coefficients. But for the BMW algebras, there exists no decomposition like \eqref{decHecke} as for the Hecke algebras, hence we do not get an explicit recursive formula for the transverse Markov trace and as a consequence nothing ensures a priori that the trace could not be computed differently and hence that the linear combination obtained as a result is unique. This question is adressed in the following lemma.

\begin{lem}\label{lemtrtrBMW}
Fix $r \in BMW_{n}$, the parameters $b, c_i \in R$, with $i = 1, \cdots, n$, such that
$$\tau_{n} (r) = b \beta + \sum_{i=1}^{n} c_i \alpha_i$$
are uniquely determined.
\end{lem}

\begin{proof}Let us fix $r \in BMW_n$. In the previous Theorem, we have seen that there exist $b, c_i \in R$ with $i= 1, \cdots,n$ such that, for all transverse Markov traces $(\tau_n)_{n \geq1}$ on $(BMW_n)_{n\geq1}$, $\tau_n(r) = b \beta + \sum_{i=1}^{n} c_i \alpha_i$. 

For any $k \in \Z_{>0}$, one can lift the basic transverse Markov traces defined in Section \label{arbbastrtr} to the tower of BMW algebras just by setting that, for any $u \in BMW_n$,
$$ \tau_n^{(k)}(u) = \tau_n^{(k)}(q_E(u)).$$
Let $j \in \{1, \cdots, n\}$, one then has $\tau_n^{(j)}(r) = b \beta + \sum_{i=1}^{n} c_i \alpha_i = c_j$ as 
$$\beta= \tau_2^{(j)}(e_1)= \tau_2^{(j)}(q_E(e_1)) = \tau_2^{(j)}(0) = 0$$
and
$$\alpha_i= \tau_i^{(j)}(1)= \delta_{i,j}.$$
As a conclusion, for all $j=1, \cdots,n$, the parameter $c_j$ is uniquely defined as $\tau_n^{(j)}(r)$.

Consider now the classical Markov trace $t_n^K$, which in turns is also a transverse Markov trace on the tower of the BMW algebras. Note that one has $t_n^K(e_1) = \delta$ which is non-zero. Then the parameter $b$ is uniquely defined as 
$$ \frac{t_n^K(r) - \sum_{i=1}^{n} \tau_n^{(i)}(r) t_i^K(1)}{t_n^K(e_1)}.$$
\end{proof}

\begin{thm*}
For any family of parameters $\{ \beta, \alpha_n\}_{n\geq 1} \subset R$, there exists a unique transverse Markov trace $(\tau_n)_{n\geq 1}$ on $(BMW_n)_{n\geq1}$ such that $\tau_n(e_1) = \beta$ and $\tau_n(1)= \alpha_n$.
\end{thm*}

\begin{proof}
If such a transverse trace exists, it is necessarily unique by the Statements \ref{thmtrtrBMW1} and \ref{lemtrtrBMW} because the chosen family of parameters determines entirely the trace.

To see that it exists, let us construct one as follows:
$$\tau_n(r)= \frac{\beta}{\delta} t_n^K(r) + \sum_{i=1}^{n} \tau_n^{(i)}(r) (\alpha_i - \beta \delta^{i-1}). $$
\end{proof}

\begin{cor}
The family $\{\left(t_n^K\right)_{n\geq1}, \left(\tau_n^{(i)}\right)_{n\geq1}, i\in \Z_{>0}\}$ forms a basis of the space of transverse Markov traces on the tower of BMW algebras $(BMW_n)_{n\geq1}$.
\end{cor}

\vspace{1cm}

\paragraph*{\bf Acknowledgements}
\ 
\newline
A.-L.T. is grateful to the Hausdorff Research Institute for Mathematics (HIM), University of Bonn, for the warm hospitality and the support as this project was completed during her visit at the HIM in fall 2017 within the Trimester Program 'Symplectic Geometry and Representation Theory'.

\bibliographystyle{alpha}
\bibliography{biblioBMW}
\vspace{0.1in}

\noindent L.P.d'A.: { \sl \small Laboratoire de Math\'ematiques de Reims, Universit\'e de Reims Champagne-Ardenne, 51687 Reims, France} 
\newline \noindent {\tt \small email: loic.poulain-dandecy@univ-reims.fr}

\noindent A.-L.T.: { \sl \small Institut f\"ur Geometrie und Topologie, Fachbereich Mathematik, Universit\"at Stuttgart, 70569 Stuttgart, Deutschland} 
\newline \noindent {\tt \small email: anne-laure.thiel@mathematik.uni-stuttgart.de}

\noindent E.W.: { \sl \small Institut Math\'ematique de Bourgogne, UMR 5584, Universit\'e de Bourgogne Franche Comt\'e, 21078 Dijon, France} 
\newline \noindent {\tt \small email: emmanuel.wagner@u-bourgogne.fr}

\end{document}